\documentclass[12pt,letterpaper]{amsart}%
\usepackage{amsmath}
\usepackage{amsfonts}
\usepackage{amssymb}
\usepackage{graphicx}
\usepackage{amsmath}
\usepackage{txfonts}
\usepackage{graphicx}
\usepackage{amsthm}
\usepackage{amsfonts}
\usepackage{amssymb}
\usepackage[utf8]{inputenc}
\usepackage{comment}
\usepackage{xcolor}
\usepackage{lipsum}
\usepackage{cancel}
\usepackage{mathrsfs}
\usepackage{ifthen}%
\providecommand{\U}[1]{\protect\rule{.1in}{.1in}}
\newif\ifcomments
\commentsfalse
\newboolean{bluecomments}
\setboolean{bluecomments}{true}
\newcommand{\bluecomment}[1]{\ifthenelse{\boolean{bluecomments}}{\textcolor{blue}{#1}}{}}
\newcounter{mycounter}
\numberwithin{mycounter}{section}
\newtheorem{theorem}[mycounter]{Theorem}
\theoremstyle{plain}

\newtheorem{claim}[mycounter]{Claim}

\newtheorem{corollary}[mycounter]{Corollary}

\newtheorem{example}{Example}[section]

\newtheorem{lemma}[mycounter]{Lemma}

\newtheorem{proposition}[mycounter]{Proposition}

\numberwithin{equation}{section}
\theoremstyle{definition}
\newtheorem{definition}[mycounter]{Definition}
\newtheorem{remark}[mycounter]{Remark}

\newcommand{\maxidx}{\mathord{\xi}} 

\newcommand{\note}[1]{\strut{\color{red}[#1]}}
\begin{document}
\title[Optimal Transport and Flows on KM metrics ]{Optimal Transport and Generalized Lagrangian Mean Curvature Flows on Kim-McCann Metrics}

\author{Arunima Bhattacharya, Micah Warren, and Daniel Weser}

\address{Department of Mathematics, Phillips Hall\\
 The University of North Carolina at Chapel Hill, NC }
\email{arunimab@unc.edu}

\address{Department of Mathematics\\
	University of Oregon, Eugene, OR 97403}
\email{micahw@uoregon.edu}

\address{Department of Mathematics, Phillips Hall\\
 The University of North Carolina at Chapel Hill, NC }
 \email{weser@unc.edu}

\begin{abstract}

We express the mean curvature flow of Lagrangian submanifolds in pseudo-Riemannian manifolds endowed with the Kim-McCann-Warren metric within the framework of generalized mean curvature flow on Kim-McCann manifolds. While generalized mean curvature flow has been studied in Kähler geometry, our work shows that techniques from para-Kähler geometry arise naturally in the Kim-McCann setting. Using this perspective, we prove that the Lagrangian condition is preserved along the flow. By identifying generalized mean curvature flow with Lagrangian mean curvature flow, we show that the Ma-Trudinger-Wang regularity theory applies to this setting. In particular, the cross-curvature positivity condition of Kim-McCann yields smoothly converging flows of Lagrangian submanifolds. Under the cross-curvature condition, any Lagrangian submanifold avoiding the cut locus converges exponentially to a stationary submanifold, which locally arises as the graph of an optimal transport map. Our framework substantiates the analogy between special Lagrangian geometry in almost Calabi–Yau manifolds and optimal transport theory in the Kim-McCann setting. In particular, we show that Kim–McCann manifolds equipped with a para-holomorphic volume form serve as the natural counterpart to almost Calabi-Yau manifolds.

\end{abstract}
\maketitle

\tableofcontents

\section{Introduction}

It was observed by Kim-McCann-Warren in \cite{KMW} that graphs of optimal transport maps are calibrated submanifolds of a conformal modification of the Kim-McCann metric \cite{KM} introduced in the study of
regularity theory of optimal transport. The Kim-McCann metric is an $(n,n)$ signature metric on $M \times \bar{M}$ with a natural Kähler form in which graphs of optimal transportation maps are space-like Lagrangian submanifolds.   While this suggests an analogy to the special Lagrangian geometry occurring in Calabi-Yau manifolds, neither the modified Kim-McCann-Warren metric \cite{KMW} nor the Kim-McCann metric \cite{KM} is Ricci-flat.   

In this paper, we make an observation that completes the analogy, opening up a broader bridge between the theory of Calabi-Yau metrics and Kim-McCann metrics.  The Kim-McCann metrics should be thought not to correspond to Calabi-Yau manifolds, but rather \textit{almost} Calabi-Yau manifolds, also called special K\"ahler manifolds by  Bryant (cf. Bryant \cite{bryant2000second} and Joyce \cite[pg 43]{joyce2001lectures}). The Kim-McCann-Warren metric then corresponds to conformal modifications of almost Calabi-Yau metrics, with a conformal factor chosen so that special Lagrangian calibrations exist. Such \textit{almost} Calabi-Yau metrics have been studied in the K\"ahler setting for well over two decades (see Goldstein \cite{Gold}) and give rise to generalized mean curvature flow, which flows towards a special Lagrangian submanifold, preserving the Lagrangian condition, even though the ambient metric is not Einstein. Note that mean curvature flow fails to preserve the Lagrangian condition when the ambient manifold is not Einstein (see Bryant \cite{Bryant}).

Here we show that the exact analogue is true in the Kim-McCann setting; the corresponding generalized mean curvature flow preserves the Lagrangian condition and is stationary on calibrated submanifolds.  This analogy is rich with interesting avenues to explore.

The Kim–McCann metric continues to be of significant interest, as reflected in the recent geometric exposition in L\'eger-Vialard \cite{LV} and the advancements in Brendle-Leg\'er-McCann-Rankin \cite{BLMR}. In this work, we build on these developments by extending the results of \cite{BLMR} from the elliptic to the parabolic setting, demonstrating that the Ma-Trudinger-Wang \cite{MTW} regularity theory ensures not only the smoothness of stationary solutions, but also the regularity of the flow.

This generalizes and combines two aspects of regularity theory. 
The first is the result of Brendle-Leg\'er-McCann-Rankin in \cite{BLMR} (see also Warren \cite{W2025} for the two-dimensional case), which shows that optimal transport maps satisfying the Ma-Trudinger-Wang condition can be proven regular using purely minimal surface techniques. Here, we generalize this from the elliptic to the parabolic setting using generalized mean curvature flow.  At the same time, we extend the parabolic Monge-Amp\`ere flow introduced in Kim-Streets-Warren \cite{KSW} to spacelike submanifolds that do not arise as the cost-exponential of a global scalar function, and show this flow is well-behaved even without the global scalar potential function.  In particular, the result can be applied to a wide variety of spacelike immersions, including non-graphical immersions.

Generalized mean curvature flow was introduced independently by Behrndt \cite{B} and Smoczyk-Wang \cite{SWAsianJMath}, with each approach highlighting different structural aspects: One centered on a torsion connection and the other on a Ricci potential. While we draw heuristic insights from these generalized flows, particularly from the para-complex setting, our analysis does not rely on their machinery. Instead, we present self-contained arguments that mirror those in the elliptic case, providing direct proofs of our main results without requiring the broader para-complex framework.
Our perspective differs from that of Chursin-Sch\"afer-Smoczyk \cite{CSS}, which, following Smoczyk-Wang \cite{SWAsianJMath}, analyzes the geometry via comparisons between different connections. In contrast, we adopt an approach more in line with Behrndt \cite{B}, emphasizing the role of calibrating forms as introduced in Kim-McCann-Warren \cite{KMW}.

We present our main results below, with formal definitions deferred to later sections to keep the introduction streamlined. 
Our first result draws inspiration from a result of Smoczyk \cite{Smo99}, subsequently extended by Behrndt \cite{B} and Smoczyk-Wang \cite{SWAsianJMath}.

\begin{theorem}
\label{thm:main1}Suppose that $M \times \bar{M}$ is a Kim-McCann manifold endowed with a para-holomorphic $n$-form $\Omega $  and $L(t)$ is a family of immersed Lagrangian submanifolds flowing by the generalized mean curvature flow.   If $L(0)$ is  Lagrangian, then so is $L(t)$ for $t>0$.   
\end{theorem}

For the precise definition of a Kim-McCann manifold, we refer the reader to Definition \ref{def:km}. The definition of the para-holomorphic $n$-form is given in Definition \ref{defom}, and the generalized mean curvature flow is defined in Definition \ref{def:gmcf}.


\smallskip
Our next result extends the elliptic regularity theory of Brendle-Leg\'er-McCann-Rankin \cite{BLMR} to the parabolic setting and generalizes the parabolic Monge-Ampère flow of Kim-Streets-Warren \cite{KSW} to spacelike submanifolds beyond cost-exponentials of global potentials.

\begin{theorem}
\label{thm:main2}Suppose that $M \times \bar{M}$ is a Kim-McCann manifold with a positive cross-curvature condition, and $L(t)$ is a family of compactly immersed Lagrangian submanifolds flowing by the generalized mean curvature flow.  If $L(t)$ stays in a compact set in the complement of the cut-locus, on $[0,T)$, then the flow extends beyond $T$. If the flow stays in a compact set avoiding the cut-locus for all time, then it converges to an immersed calibrated submanifold.  
\end{theorem}

For the precise definition of the cut locus, we refer the reader to Definition \ref{def:cut}. Here, the term “calibrated” is taken with respect to the metric introduced by Kim-McCann-Warren \cite{KMW}.

\begin{remark}
  The convergence of Lagrangian mean curvature flow in the Calabi-Yau setting has long been studied following the Thomas-Yau conjecture \cite{thomas2001special}; they conjectured that the Lagrangian mean curvature flow in a Calabi-Yau manifold should exist for all time and converge to a special Lagrangian, provided the initial Lagrangian is stable in an appropriate sense (see Joyce \cite{joyce2015conjectures} for subsequent refinements of this conjecture).  Here, we show that the cross-curvature condition provides a framework in which an analog of the Thomas-Yau conjecture holds.
\end{remark}

Our proof of Theorem \ref{thm:main1} uses an energetic method and is inspired in part by a strategy first mapped out in Smoczyk \cite{S96} (see also Behrndt \cite{B} and Wood \cite{Wood}). Necessarily, we use para-complex numbers and para-K\"ahler geometry.  

The proof of Theorem \ref{thm:main2} relies on a geometric maximum principle for a Riemannian metric on the product space along the graph, building on the work of Brendle-Leg\'er-McCann-Rankin \cite{BLMR} while extending their method to the parabolic setting. To obtain exponential convergence, we apply a Li-Yau Harnack inequality, in the spirit of Kim-McCann-Warren \cite{KMW} and Abedin–Kitagawa \cite{AK}.

Rather than aiming to construct a comprehensive dictionary linking para-complex geometry and optimal transport, our goal is more focused: To demonstrate that key results in optimal transport theory admit self-contained proofs within the framework of generalized Lagrangian mean curvature flow.  We direct the reader to Cruceanu-Fortuny-Gadea \cite{Para} for a history of para-complex numbers and para-complex geometry.  
\subsection{Relation to the Sinkhorn Algorithm}

Recent work by Deb-Kim-Pal-Schiebinger \cite{Paletal} has identified a continuous-time limit of Sinkhorn’s algorithm for entropy-regularized optimal transport as the regularization parameter tends to zero, yielding a Wasserstein mirror gradient flow for the transport cost on the space of probability measures. For smooth optimal transport plans induced by a potential, that admits sufficient regularity, this measure-valued flow can be formally rewritten in new coordinates, leading to a parabolic Monge-Ampère type flow; Berman \cite{berman2020sinkhorn} shows that Sinkhorn iterations, in a joint limit (mesh size $\rightarrow 0, \varepsilon\rightarrow 0$, iterations ~$\varepsilon^{-1}$), converge to a parabolic Monge-Ampère equation for the optimal transport potential.

From this perspective, the parabolic Monge-Ampère flow appearing in optimal transport may be viewed as a refined, geometric realization of the Sinkhorn limit, obtained after restricting to graphical Lagrangian submanifolds and passing from Wasserstein space to the Kim-McCann geometric framework. Our work operates directly at this geometric level: We interpret the parabolic Monge-Ampère equation as a generalized Lagrangian mean curvature flow on Kim-McCann manifolds, establish preservation of the Lagrangian condition, and prove long-time existence and convergence under the positive cross-curvature condition.

\subsection{Organization}

In Section 2, we develop the appropriate notions and constructions adapted from almost Calabi-Yau theory. We introduce the Lagrangian angle and the generalized mean curvature flow, and explain their relationship. In Section 3, we prove that the Lagrangian condition is preserved along the flow by showing that the integral of the Kähler form restricted to 
$L$ vanishes identically, using a Gronwall-type argument. In Section 4, we incorporate the MTW condition and present a parabolic analogue of certain results from \cite{BLMR}. Finally, in Section 5, we establish higher regularity and exponential convergence.
\bigskip

\noindent\textbf{Acknowledgments.} AB acknowledges
the support of NSF grant DMS-2350290, the Simons Foundation grant MPS-TSM-00002933, and a Bill-Guthridge fellowship from UNC-Chapel Hill. DW acknowledges
the support of the NSF RTG Grant DMS-2135998.

\section{Geometric Background}


\subsection{Kim-McCann Manifold}

\medskip

We define a Kim-McCann manifold by slightly generalizing the setting of \cite{KM}. 

\begin{definition}\label{def:km}\label{def:km}
A  \textit{Kim-McCann manifold} is a product manifold $M^n\times \bar{M}^n$ equipped with a metric $h$ satisfying the following property: for every point $(x,\bar{x}) \in M^n\times \bar{M}^n\setminus \mathscr{C}$, where $\mathscr{C}$ denotes the cut locus (defined below), there exists a function $c$ locally defined near $(x,\bar{x})$ such that for all $V \in T_{x}M^n$ and $\bar{W} \in T_{\bar{x}}\bar{M}^n$, 
\[
h(V,\bar{W}) := -\tfrac{1}{2}\,V\bar{W}c,
\]
and moreover, $h$ vanishes on pairs of vectors tangent to the same factor.
\end{definition}

\begin{definition}\label{def:cut}
We define the set $\mathscr{C}$ where the metric is not well-defined to be the \textit{cut-locus}. The term derives from the set of points on a Riemannian manifold at which the squared distance function encounters differentiability issues.
\end{definition}

We will assume the function $c$ is at least $C^4$ away from the set $\mathscr{C}$. 
 The metric $h$ can be written in the following form:
\[
h:=\frac{1}{2}\left(
\begin{array}
[c]{cc}%
0 & -c_{i\bar{s}}\\
-c_{\bar{s}i} & 0
\end{array}
\right)  .
\]



\begin{remark} 
By permitting the metric-defining function to be locally defined, our framework extends naturally to cases like the flat torus. For instance, the Kim–McCann metric derived from the squared distance remains smooth, even though the underlying cost function has a singularity; singularities that our definition is designed to bypass.
A rationally-sloped geodesic in a torus, in this case, need not encounter the cut-locus.   
\end{remark}
 Note that the form
 \[
\omega = -\frac{1}{2} c_{i\bar{s}}\, dx^{i} \wedge d\bar{x}^{s}
\]
is a symplectic form. To interpret the symplectic form as a para-K\"{a}hler form, consider the map
\[
K:T_{\left(  p,\bar{p}\right)  }M\times\bar{M}\rightarrow T_{\left(  p,\bar
{p}\right)  }M\times\bar{M}%
\]
which is represented as
\begin{equation}\label{defineK}
K_{p,\bar{p}}=I_{T_{p}M}\oplus\left(  -I\right)  _{T_{\bar{p}}\bar{M}}.
\end{equation}
For a basic introduction to para-K\"{a}hler geometry, we refer the reader to \cite{Para}. \ The
important fact is that the complex structure map $J$ satisfying $J^{2}=-I$ is
replaced by a para-complex structure map, $K$, satisfying $K^{2}=I$. \ Then one has
\[
\omega(\cdot,\cdot)=h(K\cdot,\cdot).
\]
It was noted in \cite{KM} that graphs $\left(  x,T(x)\right)  $ of optimal
transport maps are Lagrangian with respect to this symplectic form, and calibrated with respect to a conformal metric depending on the mass densities \cite{KMW}. In \cite[Claim 2.2]{W2025}, it was observed that  $K$ is parallel with respect to the
Kim-McCann metric.

\subsection{Para-complex Geometry} A para-complex valued form $\eta$ is of $\left(  n,0\right)  $ type if
\[
\mathbf{k}\cdot\eta(V_{1},..,V_{n})=\eta(KV_{1},..,V_{n})=\eta(V_{1}%
,KV_{2},..,V_{n}), \mbox{  etc.}
\]
Here $\mathbf{k}$ is an algebraic object akin to the more familiar
$\mathbf{i}$, but instead satisfies $\mathbf{k}^{2}=1.$ Define the numbers $\tau$ and $\bar{\tau}$ by
\begin{equation}\label{definetau}
    \tau = \frac{1}{2}(1+\textrm{\bf k}) \qquad \textrm{and} \qquad \overline{\tau}=\frac{1}{2}(1-\textrm{\bf k})\,,
\end{equation}
which satisfy
\begin{equation} \label{taurelations}
    \tau^2=\tau\,, \quad \overline{\tau}^{\,2}=\overline{\tau}\,, \quad \textrm{and} \quad \tau\overline{\tau}=0\,.
\end{equation}
We note that two standard notational conventions collide here. In the optimal transport literature, the bar is commonly used to denote the target manifold, while in the complex and para-complex setting, the overline denotes conjugation. The equations below necessarily employ both conventions, although the distinction should be clear from the context.

We define  
\[
\frac{\partial}{\partial{z^i}}=\tau\frac{\partial}{\partial{x^i}} +\overline{\tau}\frac{\partial}{\partial{\bar{x^i}}}%
\] and 
\[
d{z^i}=\tau d{x^i}+\overline{\tau} d{\bar{x^i}}.
\]
Noting that 
\[
d{z^1}\wedge ...\wedge d{z^n}=\tau d{x^1}\wedge ...\wedge d{x^n}+\overline{\tau} d\bar{x^1}\wedge ...\wedge d\bar{x^n} %
\]
and using the relations in (\ref{taurelations}), one can check that any para-holomorphic $(n,0)$-form can be written as 
\[  \eta= f(z) d{z}
\] (using $dz=dz^{1}\wedge...dz^{n},$ etc.) for $f$ satisfying 
 \[ 
\frac{\partial}{\partial\overline{z^i}} f = 0.
 \] 
 Writing 
 \[  
 f(x) = \tau u(x, \bar{x}) + \bar{\tau} v(x, \bar{x}) 
 \] the Cauchy-Riemann equations become
\begin{align*}
\partial_{y}u  & =0\\
\partial_{x}v  & =0.
\end{align*}
Thus 
for real valued $u$ and $v$  
\begin{equation} \label{defom}
\Omega=\tau \, u(x) dx+\overline{\tau}\,v(\bar{x})d\bar{x}%
\end{equation}
is a para-holomorphic $\left(  n,0\right)$-form. 

\subsection{Almost Calabi-Yau Manifold}

Our first goal is to demonstrate that there is indeed a strong analogy between almost Calabi-Yau manifolds and Kim-McCann manifolds.  

As motivation, we now briefly recall the classical complex setting, for example, see the work of Joyce \cite[Section 8.4]{Joyce}. Given $n\geq2$, an \textit{almost Calabi-Yau} $n$-fold is a
quadruple $(X,J,g,\Omega)$ such that $(X,J,g)$ is a compact $n$
(complex)-dimensional K\"{a}hler manifold, and $\Omega$ is a non-vanishing
holomorphic $(n,0)$-form on $X.$
In this setting, one defines a function $\psi$ (see \cite[(15)]{behrndt2008}, cf. \cite[8.24]{Joyce}) such that
\begin{equation}
e^{2n\psi}\frac{\omega^{n}}{n!}=(-1)^{n(n-1)/2}\left(  \frac{i}{2}\right)
^{n}\Omega\wedge\overline{\Omega}. \label{CY}%
\end{equation}

We would like to explore the analogy in the para-complex setting.
Recalling (\ref{defom}) we see \begin{align*}
\Omega\wedge\overline{\Omega}^{_\textrm{\hspace{.03cm}\bf k}} &  =\Big( 
\tau\,u(x) dx +
\overline{\tau}\,v(\bar{x})\,d\bar{x}%
\Big)  \wedge \Big(
\overline{\tau}\,u(x)  dx +
\tau\,v(\bar{x})\,d\bar{x}%
\Big)  \\
&  =\tau^{2}\,uv\,dx\wedge d\bar{x}+\overline{\tau}^{\,2}\,uv\,d\bar{x}\wedge dx\\
&  =\left(\tau^{2}+\left(  -1\right)^{n^{2}}\,\overline{\tau}^{\,2}\right)\,uv\,dx\wedge d\bar{x}\\
& =uv\,\mathbf{k}^{n}\,dxd\bar{x}.
\end{align*}
In particular, if 
\begin{align}
    \Omega
        \;&=\;\frac{1}{2}\big( \rho(x)dx+\bar{\rho}(\bar{x})d\bar{x}\big)
        +\frac{1}{2}\mathbf{k}\big(  \rho(x)dx-\bar{\rho}(\bar{x})d\bar{x}\big) \label{def Omega} \\
        \;&=\; \tau\rho(x)dx+\overline{\tau} \bar{\rho}(\bar{x})d\bar{x} \label{def Omega w tau}
\end{align}
one has
\begin{equation}
\Omega\wedge\overline{\Omega}^{_\textrm{\hspace{.03cm}\bf k}}=\rho(x)\bar{\rho}(\bar{x})\mathbf{k}^{n}dxd\bar
{x}.\label{omega wedge itself}%
\end{equation}

With this setup, consider a function $\psi$ that satisfies an equation similar to (\ref{CY}), replacing $i$ with $\bf{k}$
\begin{equation}
e^{2n\psi}\frac{\omega^{n}}{n!}=(-1)^{n\left(  n-1\right)  /2}\left(
\frac{\mathbf{k}}{2}\right)  ^{n}\Omega\wedge\overline{\Omega}^{_\textrm{\hspace{.03cm}\bf k}}.\label{def:psi_imp}%
\end{equation}
Now as
\[
\omega=\frac{1}{2}\sum_{i,\bar{s}}\left(  -c_{i\bar{s}}\right)  dx^{i}\wedge
d\bar{x}^{s}%
\]
one can check that
\[
\omega^{n}=\left(  \frac{1}{2}\right)  ^{n}n!(-1)^{n(n-1)/2}\det\left(
-c_{i\bar{s}}\right)  dxd\bar{x},
\]
so to satisfy our version of equation (\ref{CY}), we require
\[
e^{2n\psi}=\frac{\rho\bar{\rho}}{\det\left(  -c_{i\bar{s}}\right)  }, 
\] 
that is
\begin{equation}
\psi=\frac{1}{2n}\left(  \ln\rho+\ln\bar{\rho}-\ln\det\left(  -c_{i\bar{s}}\right)  \right).
\label{def:psi}
\end{equation}

Recall  (cf. \cite[pg 166]{Joyce}) that defining 
\[
\tilde{g} = e^{2\psi} g
\]
the (real) $n$-form $\operatorname{Re}(\Omega)$ becomes a calibration that detects special Lagrangian submanifolds in the conformal metric.  

On our side of the analogy, this would suggest that (recalling (\ref{def Omega}))
\[
\operatorname{Re}(\Omega) =\frac{1}{2}\left(  \rho(x)dx+\bar{\rho}(\bar{x})d\bar{x}\right)
\]  
would be a calibration for the metric defined by 
\begin{equation}
\tilde{h}:=\frac{1}{2}\left(  \frac{\rho(x)\bar{\rho}(\bar{x})}{\det\left(
-c_{i\bar{s}}\right)  }\right)  ^{1/n}\left(
\begin{array}
[c]{cc}%
0 & -c_{i\bar{s}}\\
-c_{\bar{s}i} & 0
\end{array}
\right)  . \label{metric}%
\end{equation}

But this is precisely the conclusion of \cite{KMW}, namely that the graph of the optimal transport map pairing $\rho$ to
$\bar{\rho}$ with cost $c$, is a volume maximizing submanifold of $M\times
\bar{M}$ with the above metric and calibration.

To continue exploring the analogy, we define a \textit{Generalized Kim-McCann manifold}.

\begin{definition}
 Let $n\geq2$. A \textit{Generalized Kim-McCann (GKM)} $n$-fold is a
quadruple $(M \times \bar{M},K,h,\Omega)$ such that $(M \times \bar{M},K,h)$ is a compact $n+n$
dimensional Kim-McCann manifold, and $\Omega$ is a non-vanishing
para-holomorphic $(n,0)$-form on $M \times \bar{M}.$
\end{definition}

Given a GKM manifold, let $\psi$ be defined implicitly by (\ref{def:psi_imp}), and let $L$ be a Lagrangian submanifold. The generalized mean curvature is then defined as
\[
\vec{H}_G := \vec{H} - n\, (\hat\nabla \psi)^\perp
\]
where $\vec{H}$ is the standard mean curvature. 
This coincides with the definition used in the almost Calabi–Yau setting. We will later show that
\[
\mathcal{L}_{\vec{H} - n\, (\hat{\nabla} \psi)^\perp} \omega = 0
\]
implying that the flow generated by this vector field preserves the Lagrangian condition. A rigorous proof of this claim will be given using the maximum principle.

We also remark that the mean curvature of $L$ with respect to the conformal metric (\ref{metric}) is
\[
e^{-2\psi} \vec{H}_G.
\]

\subsection{Lagrangian Submanifolds of Almost Calabi-Yau Manifolds}

We consider the following immersions 
$$ F : L^n \to M^n\times \bar{M}^n.$$ For an arbitrary chart on $L$ we may take $\partial_iF$ to be a basis for the tangent space in  $M \times \bar{M}$. Before proceeding, we make the following observation.

\begin{lemma} \label{graphical}
Any spacelike submanifold is locally a graphical submanifold over $M.$ 
\end{lemma}

\begin{proof}
Since the induced metric is Riemannian, we are free to choose an orthonormal chart
for $T_{(x,\bar{x})}$ at any point $(x,\bar{x})\in L$. Denote the basis at $(x,\bar{x})$ by
\[
\partial_{i}F=\left(  \xi_{i},\zeta_{i}\right)  \in T_{x}M\times T_{\bar{x}}\bar{M}.
\]
We claim that $\left\{  \xi_{i}\right\}  $ are independent: \ If not, then
we have for some nontrivial $a_{j}$
\[
\xi_{i}=\sum_{j\neq i}a_{j}\xi_{j}%
\]
in which case, we have the vector equality:%
\[
\partial_{i}F-\sum_{j\neq i}a_{j}\partial_{j}F=\left(  0,\zeta_{i}-\sum_{j\neq i}a_{j}%
\zeta_{j}\right)  .
\]
Observe that the left-hand side has length
\[
\left\Vert \partial_{i}F-\sum_{j\neq i}a_{j}\partial_{j}F\right\Vert ^{2}=1+\sum_{j\neq
i}a_{j}^{2}%
\]
whereas
\[
\left\Vert \left(  0,\zeta_{i}-\sum_{j\neq i}a_{j}\zeta_{j}\right)
\right\Vert ^{2}=0
\] a contradiction. 
It follows that $\left\{  \xi_{i}\right\}  $ do form a basis for $T_{x}M.$ The tangent space to $L$ is then transverse to $ \{\bar{x}\} \times \bar{M}$ and it follows (see \cite[Corollary 6.33]{Lee}) that L is locally a graph over $M$. 

\end{proof}

It is worth noting that if the graph $(x,T(x))$ is Lagrangian, it will be
locally described via a cost-exponential: Consider the one-form
\[
\eta= -c_{i}(x,T(x))dx^{i}.
\]
One can check that $\eta$ is closed when the graph is K\"ahler. Thus, locally, we get
\begin{equation}\label{def:cexp}
Du + D_{x}c(x, T(x)) = 0
\end{equation}
for some (local) potential $u$.  In fact, on any simply connected neighborhood on which the Lagrangian submanifold is represented as a graph, the function $u$ will exist.   
Note that a differentiable Kantorovich potential, if it exists, will satisfy (\ref{def:cexp}). However, we do not assume a priori that such a potential is available.

Thus, it is possible to invoke gradient-type graphs locally whenever a disk-shaped region projects down to a simply connected neighborhood in $M.$   
This can be accomplished by fixing a metric $m$ on $M$ and working locally within balls whose diameters are smaller than the injectivity radius with respect to $m$.

\subsection{The Lagrangian Angle}
Building on the framework introduced in \cite{behrndt2008}, we define the Lagrangian angle $\theta$ through the relation
\begin{equation}\label{theta}
    \Omega\big|_L = e^{{\bf k}\theta+n\psi}d\textrm{Vol}_L, 
\end{equation}
where
\begin{align*}
    e^{{\bf k}\theta} \;=\; \sum_{l=0}^\infty \frac{\theta^{2l}}{(2l)!} + {\bf k}\sum_{m=0}^\infty \frac{\theta^{2m+1}}{(2m+1)!} \\\;=\; \cosh\theta+{\bf k}\sinh\theta\,
    \\\;=\;  e^{\theta}\tau+e^{-\theta}%
\bar{\tau}.
\end{align*}

Note that, unlike in the Calabi–Yau setting, the angle is defined as a real number.

\begin{proposition}\label{theta is defined}
    If $L$ is a Lagrangian spacelike submanifold of a GKM, then the Lagrangian angle is well-defined.
\end{proposition}

\begin{proof}

Take a basis of an immersion $F$ that is locally defined as a graph, that is $F_{i}=\partial _{i}F=\left( E_{i},T_{i}^{s}E_{\bar{s}}\right),$  
where $E_i$  form a coordinate basis for $T_p M$ and  $ E_{\bar{s}}$ form a coordinate basis for $T_{\bar{T(p)}}\bar{M}$ for some function $T: M \to \bar{M}$.

Then recalling 
\[
\Omega =\tau \rho (x)dx+\bar{\tau}\bar{\rho}(\bar{x})d\bar{x}
\]%
we have  
\begin{eqnarray*}
\Omega _{|_{L}}\left( F_{1},...,F_{n}\right)  &=&\tau \rho (x)dx\left(
E_{1},E_{2},..,E_{n}\right) +\bar{\tau}\bar{\rho}(\bar{x})d\bar{x}%
(T_{1}^{\bar{s}}E_{\bar{s}},T_{2}^{\bar{s}}E_{\bar{s}},...) \\
&=&\tau \rho (x)+\bar{\tau}\bar{\rho}(\bar{x})\det DT.
\end{eqnarray*}%
On the other hand 
\[
dVol_{L}(F_{1},...F_{n})=\sqrt{g_{ij}}
\]%
for 
\begin{eqnarray*}
g_{ij} &=&F_{i}\cdot F_{j}=\left( E_{i},T_{i}^{s}E_{\bar{s}}\right) \cdot
\left( E_{j},T_{j}^{s}E_{\bar{s}}\right)  \\
&=&\frac{1}{2}\left( -c_{i\bar{s}}T_{j}^{s}-c_{j\bar{s}}T_{i}^{\bar{s}%
}\right) . \\
&=&(\frac{W+W^{T}}{2})_{ij}
\end{eqnarray*}%
where
\[
W_{ij}=-c_{i\bar{s}}T_{j}^{s}.
\]%
Recall that $W$ is symmetric when $L$ is Lagrangian and positive definite when $L$ is spacelike.  Thus

\[
dVol_{L}(F_{1},...F_{n})=\sqrt{\det W}=\sqrt{\det \left( -c_{i\bar{s}}\right) \det
DT}
.\]%
By (\ref{def:psi})
\[
e^{n\psi }=\sqrt{\frac{\rho (x)\bar{\rho}\left( \bar{x}\right) }{\det (-c_{i%
\bar{s}})}}
\]%
so

\begin{eqnarray*}
e^{n\psi }dVol_{L}(F_{1},...F_{n}) &=&\sqrt{\frac{\rho (x)\bar{\rho}\left( 
\bar{x}\right) }{\det (-c_{i\bar{s}})}}\sqrt{\det \left( -c_{i\bar{s}%
}\right) \det DT} \\
&=&\sqrt{\rho (x)\bar{\rho}\left( \bar{x}\right) }\sqrt{\det DT}.
\end{eqnarray*}%

Next notice that
\[
\Omega _{|_{L}}\left( F_{1},...,F_{n}\right) =\tau \rho (x)+\bar{\tau}\bar{%
\rho}(\bar{x})\det DT
\]%
is a
para-complex number, with norm equal to \[
\begin{aligned}
\left\Vert \Omega _{|_{L}}(F_{1},\ldots,F_{n}) \right\Vert_{\text{para}}^{2}
&= \left( \tau \rho(x) + \bar{\tau}\,\bar{\rho}(\bar{x}) \det DT \right)
   \left( \bar{\tau}\rho(x) + \tau\,\bar{\rho}(\bar{x}) \det DT \right) \\
&= \rho(x)\,\bar{\rho}(\bar{x}) \det DT
\end{aligned}
\]
whereas 
\begin{eqnarray*}
\left\Vert e^{n\psi }dVol_{L}(F_{1},...F_{n})\right\Vert _{\text{para}}^{2}
&=& | \sqrt{\rho (x)\bar{\rho}\left( \bar{x}\right) }\sqrt{\det DT} | ^2  \\
&=&\left\Vert \Omega _{|_{L}}\left( F_{1},...,F_{n}\right) \right\Vert _{%
\text{para}}^{2}.
\end{eqnarray*}

The following can be easily shown for para-complex numbers: Suppose that $a,b,\alpha,\beta >0$ and  
\[
\left\Vert a\tau +b\bar{\tau}\right\Vert _{\text{para}}^{2}=\left\Vert
\alpha \tau +\beta \bar{\tau}\right\Vert _{\text{para}}^{2}\neq 0
\]%
then there exists a unique unit para-complex number $e^{\mathbf{k}\theta }$ such
that 
\[
a\tau +b\bar{\tau}=e^{\mathbf{k}\theta }\left( \alpha \tau +\beta \bar{\tau}%
\right). 
\]%

Thus, we can choose a unique $\theta$ such that 
\begin{equation*}
    e^{{-\bf k}\theta}\Omega\big|_L = e^{n\psi}d\textrm{Vol}_L.
\end{equation*}  
\end{proof}  

\subsubsection{Expressing $\theta$ in terms of the graphing Jacobian}
Observe that 
\begin{align*}
e^{-\mathbf{k}\theta}\Omega  & =\left(  e^{-\theta}\tau+e^{\theta}%
\bar{\tau}\right)  \left(  \rho\tau dx+\bar{\rho}\bar{\tau}d\bar{x}\right)
\\
& =e^{-\theta}\rho\tau dx+e^{\theta}\bar{\rho}\bar{\tau}d\bar{x}%
\\  = & \frac{1}{2}\left(  e^{-\theta}\rho
dx+e^{\theta}\bar{\rho}d\bar{x}\right)
+\frac{1}{2}\left(  e^{-\theta}\rho dx-e^{\theta}\bar{\rho}d\bar{x}\right)
\mathbf{k}%
\end{align*}
using (\ref{taurelations}). 
So
\begin{align*}
e^{-\mathbf{k}\theta}\Omega_{0}  & =\frac{1}{2}\left(  e^{-\theta}\rho
dx+e^{\theta}\bar{\rho}d\bar{x}\right)  \\
& +\frac{1}{2}\left(  e^{-\theta}\rho dx-e^{\theta}\bar{\rho}d\bar{x}\right)
\mathbf{k}.%
\end{align*}  From (\ref{theta}) the above expression should be real.  The imaginary part vanishes when

\begin{equation}
e^{-\theta}\rho=e^{\theta}\bar{\rho}\frac{d\bar{x}}{dx}  \label{theta defined here}
\end{equation}
where the ratio of the $n$-forms makes sense on the spacelike $n$ submanifold. This leads to 

\[
\theta=\frac{1}{2}\ln\left(  \frac{\rho dx}{\bar{\rho}d \bar{x} }\right)=\frac{1}{2}\ln\left(  \frac{\rho}{\bar{\rho}\det DT}\right)
\] when the manifold is described as the graph of a function $T.$

\begin{remark}
    In the optimal transport case, one assumes that $\rho$ and $\bar{\rho}$ are probability densities.  For a GKM, these densities need not be probabilities, nor need they give the same mass as each other.
\end{remark}
\begin{example}
    The triple-valued map
\[
z\rightarrow z^{4/3}%
\]
on the unit circle. \ This does not locally preserve the unit arclength $ds,$
but globally pushes forward $ds$ to $ds.$ This is a special Lagrangian
submanifold of $S^1 \times S^1$ with KM metric given locally by $c(x,\bar{x}) = \frac{1}{2}d^2(x,\bar{x})$ with angle $\log(3/4)\ $.  If we changed the underlying densities to $4ds$ and $3d\bar{s}$ (giving a different GKM), the Lagrangian angle would become $0$. 
\end{example}

Given the K\"ahler form $\omega$, for any normal vector field $V$ we can define a $1$-form on $L$ via 
\[
\alpha_V(\cdot)=\omega(V, \cdot).
\]

We now relate $\theta$ to the generalized mean curvature form.

\begin{proposition}
    Let $L$ be a spacelike Lagrangian submanifold of a GKM $n$-fold, and define $\theta$ as in \eqref{theta} for $\Omega$ as in \eqref{def Omega}. Then, the identity 
    \begin{equation}\label{212}
        d\theta=\alpha_{\vec{H}-n(\hat \nabla\psi)^\perp}
    \end{equation}
    holds. In particular, the generalized mean curvature form $\theta$ is exact on $L$ and 
    \begin{equation}
        \nabla \theta=K \left({\vec{H}-n(\hat \nabla\psi)^\perp}\right).
    \end{equation}

\end{proposition}

\begin{proof}

Equation (\ref{212}) is an equation of one-forms on $T_{p}L,$ so it suffices to check this
equality against a basis for $T_{p}L.$  To that end, let $F_{i}=\partial _{i}F=\left( E_{i},T_{i}^{s}E_{\bar{s}}\right)$ be a standard basis as before for
$T_{p}L$.

Differentiate (\ref{theta}) as a tensor: that is
\begin{align}
\nabla_{F_{i}}\Omega\big|_{L}(F_{1},...,F_{n})  &  =\nabla_{F_{i}}\left(
e^{\mathbf{k}\theta+n\psi}d\text{Vol}_{L}\right) \label{215} \\
&  =\big(  \mathbf{k}d\theta+nd\psi\big)[F_{i}]
\,e^{\mathbf{k}\theta+n\psi}d\text{Vol}_{L}(F_{1},...,F_{n})  \label{216}
\end{align}
where the second line follows because the volume form is parallel along $L.$

Distinguishing between the ambient and intrinsic connections, the left-hand side can be written as follows
\begin{equation}
(\nabla_{F_{i}}\Omega)(F_{1},...,F_{n})\,=\,(\hat{\nabla}_{F_{i}}\Omega
)(F_{1},...,F_{n})\,+\,\sum_{j=1}^{n}\Omega\Big(F_{1},,...,\big(\hat{\nabla
}_{F_{i}}F_{j}\big)^{\perp},....,F_{n}\Big)\,.
\label{full deriv exp for Omega}%
\end{equation}

Recalling (\ref{defineK}) and (\ref{definetau}), consider the following operators on the tangent space to the product
manifold:\
\begin{align*}
\pi_{H} :=\frac{I+K}{2}\,, \qquad
\pi_{V} :=\frac{I-K}{2} \,,
\end{align*}
and 
\[
\spadesuit=\tau\pi_{H}+\bar{\tau}\pi_{V}.%
\]
We claim that%
\begin{equation} \label{4221}
(\hat{\nabla}_{F_{i}}\Omega)(F_{1},...,F_{n})=2nd\psi[\spadesuit F_{i}]
\Omega(F_{1},...,F_{n})
\end{equation}
and
\begin{equation} \label{4222}
\sum_{j=1}^{n}\Omega\Big(F_{1},,...,\big(\hat{\nabla}_{F_{i}}F_{j}%
\big)^{\perp},....,F_{n}\Big)\,=\mathbf{k\,}\alpha_{H}(F_{i})\Omega
(F_{1},...,F_{n}).
\end{equation}

To prove (\ref{4221}), note that the expression is tensorial, and we can  extend the vector fields $F_{i}$
from a point $(x,\bar{x})$ to constant coefficient vectors in a product neighborhood of the
ambient manifold (both vertically and horizontally), so that
$\det\!\bigl(T^{\bar s}_i\bigr)$ is constant. Then we will use the expression

\begin{equation}\label{hat nabla Omega}
    (\hat{\nabla}_{F_{i}}\Omega)(F_{1},...,F_{n})=F_{i}\Omega(F_{1},...,F_{n})-\sum_{k}\Omega(F_{1},...,\hat{\nabla}_{F_{i}}F_{k},...,F_{n}).
\end{equation}

Recalling (\ref{def Omega w tau}), we have
\begin{align*}
E_{i}\Omega(E_{1},...,E_{n})  &  =E_{i}\rho\tau\\
E_{\bar{s}}\Omega(E_{\bar{1}},...,E_{\bar{n}})  &  =E_{\bar{s}}\bar{\rho}\bar{\tau} \,
,
\end{align*}
and thus the first term in \eqref{hat nabla Omega} can be written as
\begin{equation}\label{Dip}
F_{i}\Omega(F_{1},...,F_{n})=\left(  F_{i}\rho\right) \tau dx(F_{1}%
,...,F_{n})+\left(  F_i\bar{\rho}\right)\bar{\tau} d\bar{x}(F_{1}%
,...,F_{n}). 
\end{equation}
Next, we expand
\begin{align*}
\sum_{k} dx(E_{1},\ldots,\hat{\nabla}_{E_{i}}E_{k},\ldots,E_{n})
    &= \sum_{k} \Gamma_{ik}^{k} \\
    &= c^{k\bar{s}} c_{\bar{s} k i} \\
    &= E_{i}\ln\det(-c_{k\bar{s}}) \\
    &= \pi_H(F_{i})\ln\det(-c_{k\bar{s}})
\end{align*}
and similarly for the barred directions, using \cite[Lemma 4.1, stated as Lemma \ref{KM41} below]{KM}. Therefore, for the second term of \eqref{hat nabla Omega} we obtain
\begin{align*}
-\sum_{k}\Omega(F_{1}&,...,\hat{\nabla}_{F_{i}}F_{k},...,F_{n})\\
 =&-\sum_{k}\rho\tau dx(F_{1},...,\hat{\nabla}_{F_{i}}F_{k},...,F_{n})\\
& -\sum_{k}\bar{\rho}\bar{\tau}d\bar{x}(F_{1},...,\hat{\nabla}_{F_{i}}%
F_{k},...,F_{n})\\
=&-\Big(\pi_H(F_{i}) \ln\det(-c_{k\bar{s}})\Big)\rho\tau dx(F_{1},...,F_{k},...,F_{n})\\
&  -\Big(\pi_V(F_{i})\ln
\det(-c_{k\bar{s}})\Big)\bar{\rho}\bar{\tau}d\bar{x}(F_{1},...,F_{k},...,F_{n}) .
\end{align*}
Combining this expression with (\ref{Dip}), we find 
\begin{align*}
(\hat{\nabla}_{F_{i}}\Omega)(F_{1},...,F_{n})  &  =\Big(\pi_H(F_{i})\big[  \ln\rho
-\ln\det(-c_{k\bar{s}})\big]\Big)\rho\tau\,   dx(F_{1},...,F_{n})\\
& \;\;\;\;\; +\,
\Big(\pi_V(F_{i})\big[  \ln\bar{\rho}-\ln\det(-c_{k\bar{s}})\big] \Big)\bar{\rho}\bar{\tau}\, d\bar{x}(F_{1},...,F_{n}). 
\end{align*}
It is easy to check that
\begin{align*}
\pi_H(F_{i})\left[  \ln\rho-\ln\det(-c_{k\bar{s}})\right]   &  =2n\,\pi_{H}(F_{i}%
)\psi\\
\pi_V(F_{i})\left[  \ln\bar{\rho}-\ln\det(-c_{k\bar{s}})\right]   &  =2n\,\pi_{V}%
(F_{i})\psi. 
\end{align*}
Expanding using the relations (\ref{taurelations}), we arrive at (\ref{4221}):
\begin{align*}
2n&\,\big[(\spadesuit F_{i})\psi\big]\,\Omega(F_{1},...,F_{n})\\
&  =2n\,\Big[  \tau\,\pi_{H}(F_{i})\psi+\bar{\tau}
\,\pi_{V}(F_{i})\psi\Big]  \Big(
\rho\tau dx(F_{1},...,F_{n})+\bar{\rho}\bar{\tau}d\bar{x}(F_{1},...,F_{n}%
)\Big) \\
&  =(\hat{\nabla}_{F_{i}}\Omega)(F_{1},...,F_{n}).
\end{align*}
For (\ref{4222}), recall that 
    \begin{align*}
        a_{ijk} \;&=\; h(\hat{\nabla}_{F_i}F_j,KF_k) \,, \\
        \vec{H} &=\; g^{ij}a_{ijk}(-g^{kl})KF_l \,, \\ 
        \alpha_H(\,\cdot\,) \;&=\; -g^{ij}a_{ijk}g^{kl}h(F_l,\cdot\,) \,.
    \end{align*}  Fix $j$ and compute
\begin{align}
\Omega\Big(F_{1},...,\big(\hat{\nabla}_{F_{i}}F_{j}\big)^{\perp}%
,...,F_{n}\Big)\; &  =\;\Omega\Big(F_{1},...,-a_{ijk}g^{kl}KF_{l}%
,...,F_{n}\Big)\nonumber\\
\; &  =\;-a_{ijk}g^{kl}\,\Omega(F_{1},...,KF_{l},...,F_{n})\nonumber\\
\; &  =\;-a_{ijk}g^{kl}\,\mathbf{k}\,\Omega(F_{1},...,F_{l},...,F_{n}%
)\nonumber\\
\; &  =\;-a_{ijk}g^{kl}\,\delta_{lj}\,\mathbf{k}\,\Omega(F_{1},...,F_{n}%
)\,,\nonumber
\end{align}
so that, summing over $j$, we obtain
\begin{align}
\sum_{j=1}^{n}\Omega\Big(F_{1},...,\big(\hat{\nabla}_{F_{i}}F_{j}\big)^{\perp
},...,F_{n}\Big)\; &  =\;-\sum_{j=1}^{n}a_{ijk}g^{kl}\,\delta_{lj}%
\,\mathbf{k}\,\Omega(F_{1},...,F_{n})\nonumber\\
\; &  =\;-a_{ijk}g^{kj}\,\mathbf{k}\,\Omega(F_{1},...,F_{n})\nonumber\\
\; &  =\;\alpha_{\vec{H}}(F_{i})\,\text{\textbf{k}}\,\Omega(F_{1}%
,...,F_{n})\,,\label{mc form part of der Omega}%
\end{align}
where we used the fact that $a_{ijk}$ is fully symmetric in all three indices
because $L$ is Lagrangian; this gives (\ref{4222}).

Now combining (\ref{215} - \ref{4222}), we get 
\[
\Big(  \mathbf{k}d\theta[F_i]+nd\psi[F_i]\Big)\,  \Omega
(F_{1},...,F_{n})=\Big(2nd\psi[\spadesuit F_{i}]+\mathbf{k\,}\alpha_{H}(F_{i})\Big)\,\Omega(F_{1},...,F_{n}).
\]
Expanding the definition of $\spadesuit,$
\[
\spadesuit=\frac{I+\mathbf{k}K}{2}\,,%
\]
we may compare real and non-real parts to obtain
\begin{align*}
d\theta[F_{i}]  &  =nd\psi[KF_{i}]+\mathbf{\,}\alpha_{H}(F_{i})\,.
\end{align*}
Finally, we observe
\begin{align*}
\omega(n(\hat\nabla\psi)  ^{\perp},F_{i})  &  =-h(n(  \hat\nabla
\psi)^{\perp},KF_{i})\\
&  =-nh(\hat\nabla\psi,KF_{i})\\
&=-nd\psi(KF_{i})\,
\end{align*}
and hence we conclude
\[
d\theta|_{L}(F_i)=-\alpha_{n\left(  \nabla\psi\right)  ^{\perp}}(F_i)+\mathbf{\,}%
\alpha_{H}(F_{i})
\]
as desired.
\end{proof}

\section{\label{sect:main}Short-time Existence and Preservation of Lagrangian Condition}

We rely on standard results for mean curvature and generalized mean curvature flow in the Riemannian setting to establish short-time existence. This is well known in the Riemannian case (see \cite[Prop 3.2]{SmoczykSurvey}), and an analogous result for spacelike surfaces in pseudo-Riemannian manifolds appears in \cite[Prop 5.1]{LS}. The extension to generalized mean curvature flow involves only lower-order terms, so short-time existence follows by standard arguments (cf. \cite[Proof of Theorem 1]{SWAsianJMath}).

The following formulas allow us to perform an energetic argument to show that the Lagrangian condition is preserved. Necessarily, these computations need to be in the setting of totally real submanifolds (those containing no para-complex planes), which is an open set containing Lagrangian submanifolds.

We briefly recall the following definition for the reader's convenience: 

\begin{definition}
    A \textit{totally real submanifold} $L$ in the para-K\"{a}hler setting is a submanifold whose tangent space does not contain any para-complex planes; that is, for every $V\in T_{p}L$ we have $KV\notin T_{p}L.$
\end{definition}

For our purposes, it is convenient to consider immersed submanifolds with flows determined locally. 
 Namely,  given some abstract manifold $L_0$, consider a family of immersions  
$$ F: L_0 \times [0,t_0) \rightarrow M\times\bar{M} $$ governed by a flow that is determined locally by geometric quantities in $L_0$. 

\begin{definition}\label{def:gmcf}
An immersed submanifold of a GKM manifold flows by \textit{generalized mean curvature flow} if %
\begin{equation}\label{GKMMCF}
\left(  \frac{\partial F}{\partial t}\right)  ^{\perp}=\vec H-n\left(  \hat\nabla
\psi\right)  ^{\perp}%
\end{equation}
where $\psi$ is as defined in (\ref{def:psi}).
\end{definition}

We note here that the parabolic flow discussed in \cite{KSW} provides a
vertical flow of Lagrangian submanifolds. One can check that the normal projection of
the flow is the generalized mean curvature flow.  This flow requires the map $T$ \cite{KSW} to be globally described as a cost-exponential of a potential.

In this section, we will show that the Lagrangian condition is preserved.   The mean curvature flow itself evolves in the space of totally real submanifolds.  We use a Gronwall type argument to show that the integral of the K\"ahler form must remain 0 for any short time on which the flow exists.  

We also recall \cite[Lemma 4.1]{KM}.
\begin{lemma}[Riemann curvature tensor and Christoffel symbols] \label{KM41}
Use a non-degenerate cost $c \in C^{4}(N)$ to define a pseudo-metric on the domain 
$N \subset M \times \bar M$. In local coordinates $x^1,\ldots,x^n$ on $M$ and 
$x^{\bar 1},\ldots,x^{\bar n}$ on $\bar M$, the only non-vanishing Christoffel symbols are
\begin{equation}\label{eq:Gamma}
\Gamma_{ij}^{\phantom{ij}m} = c^{m\bar k} c_{kij}
\qquad\text{and}\qquad
\Gamma_{i\bar j}^{\phantom{i\bar j}\bar m} = c^{\bar m k} c_{k i \bar j}.
\end{equation}
Furthermore, the components of the Riemann curvature tensor vanish except when the
number of barred and unbarred indices is equal, in which case the value of the component
can be inferred from $R_{ij \bar k \bar l} = 0$ and
\begin{equation}\label{eq:Rijkl}
2 R_{i\bar j \bar k l}
  = c_{i\bar{j} \bar{k} l} - c_{l i \bar f} c^{ a \bar f} c_{ a \bar j \bar k}.
\end{equation}
\end{lemma}

\subsection{Special Coordinates}
We begin by developing the coordinates we will need in this section and in the sequel.

\begin{lemma}[Special coordinates of type I] \label{srecial coords lemma}
Suppose that $L$ is locally a graph over $M$ in $M\times \bar{M}$ near $(p,%
\bar{p})$ for $\bar{p}=T(p).$\ Given any coordinate system on $M$, it is
possible to choose coordinates for $\bar{M}$ near $\bar{p}$ such that  
\[
c_{i\bar{s}}\left( p,\bar{p}\right) =-\delta _{i\bar{s}}
\]%
and 
\[
c_{\bar{s}\bar{q}i}(p,\bar{p})=0\text{ for all }\bar{s},\bar{q},i.
\]
\end{lemma}

\begin{proof}
Fixing the chart for $M$, and taking any arbitrary chart for $\bar{M}$
(assuming $(p,\bar{p})$ $\rightarrow (0,0)$) compute 
\[
Q_{s}^{i}:=-c_{i\bar{s}}(0,0)
\]%
and consider first a linear change of coordinates 
\[
y^{s}\left( z^{1},...z^{n}\right) =\left( Q^{-1}\right) _{p}^{s}z^{p}.
\]%
One can then check that 
\[
\partial _{z^{p}}\partial _{x^{i}}c(0,0)=c_{i\bar{s}}\frac{\partial y^{s}}{%
\partial z^{p}}=c_{i\bar{s}}\left( Q^{-1}\right) _{p}^{s}=-\delta _{ip}.
\]%
Next, we modify these slightly to eliminate the particular third derivatives.
\ \ 

Start by assuming $c_{i\bar{s}}=-\delta _{i\bar{s}}$ holds at the origin for
coordinates $\left( x,\bar{y}\right) $ and set  
\[
y^{\bar{s}}=z^{\bar{s}}+\frac{1}{2}c_{s\bar{m}\bar{p}}(0,0)z^{\bar{p}}z^{%
\bar{m}}
\]%
which is a diffeomorphism in a small neighborhood. \ Here%
\[
c_{s\bar{m}\bar{p}}(0,0):=\frac{\partial ^{3}c}{\partial y^{\bar{m}}\partial
y^{\bar{p}}\partial x^{s}}(0,0).
\]
Note that 
\[
\frac{\partial y^{\bar{s}}}{\partial z^{\bar{p}}}=\delta _{\bar{p}}^{\bar{s}%
}+c_{s\bar{m}\bar{p}}(0,0)z^{\bar{m}}
\]%
\[
\frac{\partial ^{2}y^{\bar{s}}}{\partial z^{\bar{q}}\partial z^{\bar{p}}}%
=c_{s\bar{q}\bar{p}}(0,0).
\]%
\  Then compute  
\[
\frac{\partial ^{2}c(x,\bar{z})}{\partial z^{\bar{p}}\partial x^{i}}=\frac{%
\partial ^{2}c(x,\bar{y}(\bar{z}))}{\partial y^{\bar{s}}\partial x^{i}}\frac{%
\partial y^{\bar{s}}}{\partial z^{\bar{p}}}(x,\bar{z})
\]%
and take one more derivative
\begin{eqnarray*}
\frac{\partial ^{3}c(x,\bar{z})}{\partial z^{\bar{q}}\partial z^{\bar{p}%
}\partial x^{i}}(0,0) &=&\frac{\partial ^{3}c}{\partial y^{\bar{m}}\partial
y^{\bar{s}}\partial x^{i}}(0,0)\frac{\partial y^{\bar{m}}}{\partial z^{\bar{q%
}}}\frac{\partial y^{\bar{s}}}{\partial z^{\bar{p}}}(0,0)+\frac{\partial
^{2}c}{\partial y^{\bar{s}}\partial x^{i}}(0,0)\frac{\partial ^{2}y^{\bar{s}}%
}{\partial z^{\bar{q}}\partial z^{\bar{p}}}(0,0) \\
&=&c_{\bar{m}\bar{s}i}(0,0)\delta _{\bar{q}}^{\bar{m}}\delta _{\bar{p}}^{%
\bar{s}}+(-\delta _{i\bar{s}})c_{sq\bar{p}}(0,0))\ =c_{\bar{q}\bar{p}i}-c_{i%
\bar{q}\bar{p}}=0.
\end{eqnarray*}
\end{proof}

To be clear, we define \textit{Special coordinates of type I} to be a coordinate system where we have taken an arbitrary coordinate system on $M$ and chosen a chart for $\bar{M}$ so that the conclusions of Lemma \ref{srecial coords lemma} hold.  

We will also introduce \textit{Special coordinates of type II}, defined as follows.

\begin{corollary}[Special coordinates of type II]\label{IIcoords}
Suppose that $L$ is locally a graph over $M$ in $M\times \bar{M}$ near $(p,%
\bar{p})$ for $\bar{p}=T(p).$ We may choose normal coordinates for $M$ with
respect to the induced metric on $L$ and a coordinate system for $\bar{M}$
such that at $\left( p,\bar{p}\right) $ we have%
\begin{eqnarray*}
g_{ij}(p) &=&\delta _{ij} \\
\partial_{x^{k}}g_{ij}(p) &=&0 \\
c_{i\bar{s}}(p,\bar{p}) &=&-\delta _{i\bar{s}} \\
\hat{R}_{i\bar{s}\bar{p}j} &=&\frac{1}{2}c_{i\bar{s}\bar{p}j} \\
T_{i}^{\bar{j}} &=&\delta _{i}^{\bar{j}}-\omega |_{_{L}ij}.
\end{eqnarray*}
\end{corollary}

\begin{proof}
Inspecting the expression in Lemma \ref{KM41}
expression, eliminating the third derivatives with 2 barred indices, eliminating the second term, and we are left
with  $\frac{1}{2}c_{i\bar{s}\bar{p}j}.$ \ \ Next, noting that 
\begin{eqnarray*}
g_{ij} &=&F_{i}\cdot F_{j} \\
&=&\left( E_{i}+T_{i}^{s}E_{\bar{s}}\right) \cdot \left( E_{j}+T_{j}^{\bar{p}%
}E_{\bar{p}}\right)  \\
&=&\frac{-1}{2}\left( c_{i\bar{s}}T_{j}^{s}+c_{j\bar{s}}T_{i}^{s}\right)  \\
&=&\frac{1}{2}\left( T_{j}^{\bar{\imath}}+T_{i}^{\bar{j}}\right) 
\end{eqnarray*}%
and 
\begin{eqnarray*}
\omega (F_{i},F_{j}) &=&h(KF_{i},F_{j}) \\
&=&\left( E_{i}-T_{i}^{\bar{s}}E_{\bar{s}}\right) \cdot \left( E_{j}+T_{j}^{%
\bar{p}}E_{\bar{p}}\right)  \\
&=&-\frac{1}{2}T_{i}^{\bar{j}}+\frac{1}{2}T_{j}^{\bar{\imath}}
\end{eqnarray*}%
we have  
\[
g_{ij}=T_{i}^{\bar{j}}+\omega _{ij}
\]%
or  
\[
T_{i}^{\bar{j}}=\delta _{ij}-\omega _{ij}.
\]
\end{proof}

\subsection{Totally Real Submanifolds}

Next, we need to develop our setup in the totally real setting. 
First, given a totally real submanifold $L$, we define
\begin{align*}
\tilde{K}    :T_{p}L\rightarrow\left(  T_{p}L\right)  ^{\perp}\\
 \tilde{K}(X)    =\left(  K(X)\right)  ^{\perp}.
\end{align*}

Taking $\left\{  F_{i}\right\}  $ to be a basis of the tangent space of
$L^{n},$ one can check that defining%
\[
\omega_{i}^{l}=\omega_{ij}g^{jl}%
\]
for
\[
\omega_{ij}=\omega(F_{i},F_{j})
\]
and induced metric
\[
g_{ij}=h(F_{i},F_{j}),
\]
we get
\begin{equation}
\label{define Ktwiddle}
\tilde{K}(F_{i})=K(F_{i})-\omega_{i}^{m}F_{m}.
\end{equation}
Notice that
\begin{align*}
K\tilde{K}  &  =K^{2}(F_{i})-\omega_{i}^{l}KF_{l}\\
&  =F_{i}-\omega_{i}^{l}KF_{l}.
\end{align*}
As always, we assume a K\"{a}hler condition%
\[
\omega=h(K\cdot,\cdot).
\]
Note that in a para-K\"{a}hler manifold, we have\
\[
h(KX,Y)=-h(K^{2}X,KY)=-h(X,KY).
\]
We will also use the (negative definite) metric on $\left\{  \tilde{K}%
F_{i}\right\}  $ which forms a basis for the normal space \
\[
\eta_{ij}=h(\tilde{K}F_{i},\tilde{K}F_{j}).
\]

Next, we state the following proposition. The proof will be deferred to subsection \ref{pf}, where it will be established through a series of lemmas.
\begin{proposition}\label{evl4}
Suppose that $L$ is flowing by generalized mean curvature flow on a time
interval $\left[ 0,t_{0}\right] .$ \ \ There are bounded $G,G^{\prime }$
quantities such that 
\begin{equation*}
d\left( \alpha _{\vec{H}}-n\alpha _{\left( \nabla \psi \right) ^{\perp
}}\right) =-dd|_{L}^{\ast }\omega +\omega\ast G+\nabla \omega \ast G^{\prime }.
\end{equation*}%
Here $G,G^{\prime }$ are not a priori bounded, but are bounded depending on the particular
flow and the time interval, and $*$ refers to some geometric contraction. 
\end{proposition}

Assuming Proposition \ref{evl4}, we now prove the main result of this section.

\begin{proposition}
Suppose that $L$ is flowing by generalized mean curvature flow on a time
interval $\left[ 0,t_{0}\right] .$ \ Then 
\begin{equation*}
\frac{d}{dt}\int \left\Vert \omega |_{L}\right\Vert _{g}^{2}dV_{g}(t)\leq
C(G,G^{\prime },t_{0})\left\Vert \omega |_{L}\right\Vert _{g}^{2}(t).
\end{equation*}
\end{proposition}

\begin{proof}
Note that along the flow, 
\begin{equation*}
\frac{dF}{dt}=\vec{H}-n\left( \nabla \psi \right) ^{\perp }
\end{equation*}%
we have  
\begin{eqnarray*}
\frac{d}{dt}\omega  &=&\mathcal{L}_{\vec{H}-n\left( \nabla \psi \right)
^{\perp }}\omega =d\left( \iota _{\vec{H}-n\left( \nabla \psi \right)
^{\perp }}\omega \right) +\iota _{\vec{H}-n\left( \nabla \psi \right)
^{\perp }}d\omega  \\
&=&d\left( \alpha _{\vec{H}}-n\alpha _{\left( \nabla \psi \right) ^{\perp
}}\right)  \\
&=&-dd|_{L}^{\ast }\omega +\omega \ast G+\nabla \omega \ast G^{\prime }.
\end{eqnarray*}%
Thus 
\begin{equation*}
\frac{d}{dt}\int \left\Vert \omega |_{L}\right\Vert _{g}^{2}dV_{g}(t)=\int
2\langle \omega ,\frac{d}{dt}\omega \rangle _{g}dV_{g}(t)+\left\Vert \omega
|_{L}\right\Vert _{g}^{2}\frac{d}{dt}dV_{g}(t).
\end{equation*}%
Now 
\begin{eqnarray*}
\int 2\langle \omega ,\frac{d}{dt}\omega \rangle _{g}dV_{g}(t) &=&\int
2\langle \omega ,-dd|_{L}^{\ast }\omega +\omega \ast G+\nabla \omega \ast
G^{\prime }\rangle _{g}dV_{g}(t) \\
&=&-2\int \langle d|_{L}^{\ast }\omega ,d|_{L}^{\ast }\omega \rangle
dV_{g}(t)+\int 2\langle \omega ,\omega \ast G+\nabla \omega \ast G^{\prime
}\rangle _{g}dV_{g}(t) \\
&\leq &-2\int \left\Vert d|_{L}^{\ast }\omega \right\Vert
^{2}dV_{g}(t)+C_{1}(G)\int \left\Vert \omega |_{L}\right\Vert ^{2}dV_{g}(t)
\\
&&+\frac{C_{2}(G^{\prime })}{\varepsilon }\int \left\Vert \omega
|_{L}\right\Vert ^{2}dV_{g}(t)+\varepsilon \int \left\Vert \nabla \omega
|_{L}\right\Vert ^{2}dV_{g}(t).
\end{eqnarray*}%

\bigskip By The Weitzenb\"{o}ck Formula \cite[Theorem 9.8, see
Definition 9.6, Proposition 9.7 and (2.1.26, 2.1.28, 9.3.1, 1.0.3)\ ]{MMMT}
we know that 
\begin{equation*}
-\nabla ^{\ast }\nabla \omega =-dd^{\ast }\omega -\mathfrak{Ric(\omega )}
\end{equation*}%
where $\mathfrak{Ric}$ is a linear operator $\Lambda ^{2}(L)\rightarrow
\Lambda ^{2}(L)$ with coefficients that depend linearly on the Riemannian
curvature tensor of $\left( L,g\right) $:

\begin{equation*}
\mathfrak{Ric=}\sum_{i,j,k,l}R_{ijkl}dx_{i}\wedge (dx_{j}\vee (dx_{k}\wedge
(dx_{l}\vee \cdot ))).
\end{equation*}%

 Next
\begin{eqnarray*}
\int \left\Vert \nabla \omega |_{L}\right\Vert ^{2}dV_{g} &=&\int \langle
\nabla ^{\ast }\nabla \omega ,\omega \rangle _{g}dV_{g} \\
&=&\int \langle dd^{\ast }\omega +\mathfrak{Ric} (\omega) ,\omega \rangle _{g}dV_{g} \\
&\leq &\int \left\Vert d|_{L}^{\ast }\omega \right\Vert
^{2}dV_{g}(t)+C_{3}(\mathfrak{Ric})\int \left\Vert \omega |_{L}\right\Vert ^{2}dV_{g}(t).
\end{eqnarray*}%
Putting this together, we get 
\begin{equation*}
\frac{d}{dt}\int \left\Vert \omega |_{L}\right\Vert _{g}^{2}dV_{g}(t)\leq
-\int \left\Vert d|_{L}^{\ast }\omega \right\Vert ^{2}dV_{g}(t)+C_{4}\int
\left\Vert \omega |_{L}\right\Vert ^{2}dV_{g}(t).
\end{equation*}
\end{proof}

\begin{corollary}
Suppose that $L$ is flowing by generalized mean curvature flow on a
time interval $\left[ 0,t_{0}\right] $ and $\omega |_{L}\equiv 0$ when $t=0.$
\ Then $\omega |_{L}\equiv 0$ for all $t$ in the interval. \ 
\end{corollary}

\begin{proof}
This follows from applying a Gronwall argument: \begin{equation*}
\frac{d}{dt} \left(e^{-C_4 t}\int \left\Vert \omega |_{L}\right\Vert _{g}^{2}dV_{g}(t) \right)\leq
0.
\end{equation*} 
\end{proof}

\subsection{
Differential of the Mean Curvature Form in the Totally Real Setting}\label{pf}

In this subsection, we prove Proposition \ref{evl4}.
\begin{lemma}
The mean curvature form is related to the codifferential (on the
submanifold)\ via\
\[
\alpha_{\vec{H}}=-d|_{L}^{\ast}\omega+\beta
\]
where
\[
\beta(F_{k}) = g^{ij}\,\omega\!\left(F_{i},\,A\!\left(F_{j},F_{k}\right)\right).
\]

\end{lemma}
\begin{proof}
Using
\[
d|_{L}^{\ast}\omega(\cdot) = -\,g^{ij}\,\nabla_{F_{i}}\omega(F_{j},\cdot)
\]
we compute on both the ambient manifold and the submanifold:
\[
\hat{\nabla}_{F_{i}}\omega(F_{j},\cdot)
= F_{i}\omega(F_{j},\cdot)
  - \omega\!\left(\hat{\nabla}_{F_{i}}F_{j},\,\cdot\right)
  - \omega\!\left(F_{j},\,\hat{\nabla}_{F_{i}}(\cdot)\right),
\]
\[
\nabla_{F_{i}}\omega(F_{j},\cdot)
= F_{i}\omega(F_{j},\cdot)
  - \omega\!\left(\nabla_{F_{i}}F_{j},\,\cdot\right)
  - \omega\!\left(F_{j},\,\nabla_{F_{i}}(\cdot)\right).
\]

Taking the difference and using that 
\(\hat{\nabla}_{F_{i}}\omega(F_{j},\cdot)=0\) (since \(\omega\) is parallel), we get
\begin{align*}
\nabla_{F_{i}}\omega(F_{j},\cdot)
&= \omega\!\left(\hat{\nabla}_{F_{i}}F_{j} - \nabla_{F_{i}}F_{j},\,\cdot\right)
 + \omega\!\left(F_{j},\left(\hat{\nabla}_{F_{i}} - \nabla_{F_{i}}\right)(\cdot)\right) \\
&= \omega\!\left(A(F_{i},F_{j}),\,\cdot\right)
  + \omega\!\left(F_{j},\,A(F_{i},\cdot)\right).
\end{align*}

Tracing in \(i,j\) gives
\begin{align*}
d|_{L}^{\ast}\omega(F_{k})
&= -\,g^{ij}\,\omega\!\left(A(F_{i},F_{j}),\,F_{k}\right)
   - g^{ij}\,\omega\!\left(F_{j},\,A(F_{i},F_{k})\right) \\
&= -\,\omega(\vec{H},F_{k})
   - g^{ij}\,\omega\!\left(F_{j},\,A(F_{i},F_{k})\right),
\end{align*}
as desired.
\end{proof}


\begin{lemma} Suppose that $L$ evolves by mean curvature flow and is totally real. For any fixed $t \in [0,t_0]$, \[
d\beta=\hat{R}ic(\cdot,K\cdot)+\omega|_{L}\ast G
\]
where $\omega|_L \ast G$ denotes a contraction of terms, each containing at least one factor of $\omega|_L$, while all remaining factors are bounded by constants depending only on the given flow on $[0,t_0]$.
\end{lemma}

\begin{proof}
We take normal coordinates for the induced metric at the point and time-slice. We compute
\begin{eqnarray*}
\beta(F_{k})
&=& g^{ij}\,\omega\!\left(F_{i},\,A(F_{j},F_{k})\right) \\[4pt]
&=& g^{ij}\,h\!\left(KF_{i},\,A(F_{j},F_{k})\right) \\[4pt]
&=& g^{ij}\,h\!\left(\tilde{K}F_{i},\,A(F_{j},F_{k})\right) \\[4pt]
&=& g^{ij}\,h\!\left(KF_{i},\,A(F_{j},F_{k})\right)
   \;-\;
   g^{ij}\,\omega_{i}^{p}\,
   h\!\left(F_{p},\,A(F_{j},F_{k})\right).
\end{eqnarray*}
The third line follows since the second fundamental form is normal. The fourth line follows from the expression for $\tilde{K}$ given in (\ref{define Ktwiddle}).
 Notice the last term pairs tangential with normal, so vanishes.  Now
 \begin{eqnarray*}
F_{k}\beta(F_{l})
&=& g^{ij}\,F_{k}\,h\!\left(KF_{i},\,A(F_{j},F_{l})\right) \\[6pt]
&=& g^{ij}\,h\!\left(K\,\bar{\nabla}_{F_{k}}F_{i},\,A(F_{j},F_{l})\right)
   \;+\;
   g^{ij}\,h\!\left(KF_{i},\,\bar{\nabla}_{F_{k}}A(F_{j},F_{l})\right).
\end{eqnarray*}
To be clear, we are differentiating $A(F_{j},F_{k})$ as an ambient vector field, rather than as a tensor or as a section of the normal bundle.

If we are assuming normal coordinates for the induced metric at a point, the vector 
$\bar{\nabla}_{F_{k}}F_{i}$ is the second fundamental form, and is normal.
\ That is (at the point)\  
\[
\bar{\nabla}_{F_{k}}F_{i}=A(F_{k},F_{i})\cdot \tilde{K}F_{m}\eta ^{ml}\tilde{%
K}F_{l}=a_{kim}\eta ^{ml}\tilde{K}F_{l}
\]%
and the first term becomes \begin{eqnarray*}
g^{ij}\, h\!\left( K \bar{\nabla}_{F_{k}} F_{i},\, A(F_{j}, F_{l}) \right)
&=&
a_{kim}\, \eta^{ml}\, g^{ij}\,
h\!\left( K \tilde{K} F_{l},\, A(F_{j}, F_{l}) \right)
\\[6pt]
&=&
a_{kim}\, \eta^{ml}\, g^{ij}\,
h\!\left( K\big(KF_{l} - \omega_{l}^{p} F_{p}\big),\, A(F_{j},F_{l}) \right)
\\[6pt]
&=&
a_{kim}\, \eta^{ml}\, g^{ij}\,
\Big[
\, h\!\left( KF_{l},\, A(F_{j}, F_{l}) \right)
\;-\;
\omega_{l}^{p}\, h\!\left( K F_{p},\, A(F_{j}, F_{l}) \right)
\Big].
\end{eqnarray*}
The first term of this expression is normal paired with tangential, so it
vanishes, and the second term is of the form $\omega \ast G.$
Computing
\[
F_{k}\beta (F_{l})-F_{l}\beta (F_{k})=\omega \ast G+g^{ij}h\left( KF_{i},%
\bar{\nabla}_{F_{k}}A\left( F_{j},F_{l}\right) -\bar{\nabla}%
_{F_{l}}A(F_{j},F_{k})\right) .
\]%
Next we get
\begin{eqnarray*}
\bar{\nabla}_{F_{k}}A\left( F_{j},F_{l}\right) -\bar{\nabla}_{F_{l}}A(F_{j},F_{k})
&=&
\bar{\nabla}_{F_{k}}\left( \bar{\nabla}_{F_{j}}F_{l}-\nabla _{F_{j}}F_{l}\right)
-\bar{\nabla}_{F_{l}}\left( \bar{\nabla}_{F_{j}}F_{k}-\nabla _{F_{j}}F_{k}\right)
\\
&=&
\bar{\nabla}_{F_{k}}\bar{\nabla}_{F_{j}}F_{l}
-\bar{\nabla}_{F_{l}}\bar{\nabla}_{F_{j}}F_{k}
-\bar{\nabla}_{F_{k}}\left( \Gamma _{jl}^{m}F_{m}\right)
+\bar{\nabla}_{F_{l}}\left( \Gamma _{jk}^{m}F_{m}\right)
\\
&=&
\bar{\nabla}_{F_{k}}\bar{\nabla}_{F_{j}}F_{l}
-\bar{\nabla}_{F_{l}}\bar{\nabla}_{F_{j}}F_{k}
-\;F_{k}\Gamma _{jl}^{m}F_{m}
-\Gamma _{jl}^{m}A(F_{k},F_{m})
\\
&&\qquad
+\;F_{l}\Gamma _{jk}^{m}F_{m}
+\Gamma _{jk}^{m}A(F_{l},F_{m})
\\
&=&
\bar{\nabla}_{F_{k}}\bar{\nabla}_{F_{j}}F_{l}
-\bar{\nabla}_{F_{l}}\bar{\nabla}_{F_{j}}F_{k}
-\;F_{k}\Gamma _{jl}^{m}F_{m}
+\;F_{l}\Gamma _{jk}^{m}F_{m}
\end{eqnarray*}
again using that Christoffel symbols vanish at the point. \ 

So 
\begin{eqnarray*}
&&h\left( KF_{i},\bar{\nabla}_{F_{k}}A\left( F_{j},F_{l}\right) -\bar{\nabla}%
_{F_{l}}A(F_{j},F_{k})\right)  \\
&=&h(\bar{\nabla}_{F_{k}}\bar{\nabla}_{F_{j}}F_{l}-\bar{\nabla}_{F_{l}}\bar{%
\nabla}_{F_{j}}F_{k},KF_{i})+h(-F_{k}\Gamma _{jl}^{m}F_{m}+F_{l}\Gamma
_{jk}^{m}F_{m},KF_{i}) \\
&=&h(\bar{\nabla}_{F_{k}}\bar{\nabla}_{F_{j}}F_{l}-\bar{\nabla}_{F_{l}}\bar{%
\nabla}_{F_{j}}F_{k},KF_{i})+(-F_{k}\Gamma _{jl}^{m}+F_{l}\Gamma
_{jk}^{m})\omega _{im}.
\end{eqnarray*}

The second term is $\ast \omega $, so we focus on the first. Using torsion
free and vanishing of Lie brackets, we get
\begin{eqnarray*}
h(\bar{\nabla}_{F_{k}}\bar{\nabla}_{F_{j}}F_{l}-\bar{\nabla}_{F_{l}}\bar{%
\nabla}_{F_{j}}F_{k},KF_{i}) &=&h(\bar{\nabla}_{F_{k}}\bar{\nabla}%
_{F_{l}}F_{j}-\bar{\nabla}_{F_{l}}\bar{\nabla}_{F_{k}}F_{j}-\bar{\nabla}%
_{[F_{k},F_{l}]}F_{j},KF_{i}) \\
&=&\hat{R}(F_{k},F_{l},F_{j},KF_{i}) .
\end{eqnarray*}%
So
\[
d\beta (F_{k},F_{l})=g^{ij}\tilde{R}(F_{k},F_{l},F_{j},KF_{i})+ \omega *G.
\]

The proof of the Lemma is complete modulo the following claim: 
\end{proof}

\begin{claim}
\begin{eqnarray*}
g^{ij}\hat{R}(F_{k},F_{l},F_{j},KF_{i}) &=&\hat{R}ic(F_{k},KF_{l})+%
\omega \ast  \\
&=&\sum_{i}\left( C(k,\bar{l},\bar{\imath},i)-C(l,\bar{k},\bar{\imath}%
,i)\right) +\omega \ast  G
\end{eqnarray*}%
where the notation
\begin{equation*}
C(k,\bar{l},\bar{\imath},i)=c_{k\bar{l}\bar{\imath}i}
\end{equation*}%
is computed at the point in special coordinates (recalling Lemma \ref{KM41} and Corollary \ref{IIcoords}), is given the symmetries of the curvature tensor, and represented like this to be lighter on the
eyes. \
\end{claim}
\begin{proof}

Choose an induced orthonormal basis (at the point) for the submanifold
which is described at a point via 
\begin{equation*}
F_{l}=E_{l}+T_{l}^{\bar{s}}E_{\bar{s}}.
\end{equation*}%
In this case, assuming special coordinates type II, we get (\ref{IIcoords})

\begin{eqnarray*}
g^{ij}\hat{R}(F_{k},F_{l},F_{j},KF_{i})
&=& \sum_{i} \hat{R}\!\left(
E_{k}+T_{k}^{\bar{s}}E_{\bar{s}},\,
E_{l}+T_{l}^{\bar{p}}E_{\bar{p}},\
E_{i}+T_{i}^{\bar{r}}E_{\bar{r}},\,
E_{i}-T_{i}^{\bar{q}}E_{\bar{q}}
\right)
\\
&=& \sum_{i} \hat{R}\!\left(
E_{k}+E_{\bar{k}},\,
E_{l}+E_{\bar{l}},\,
E_{i}+E_{\bar{\imath}},\,
E_{i}-E_{\bar{\imath}}
\right)
+ \omega \ast G.
\end{eqnarray*}
Using  
\begin{equation*}
\hat{R}_{kl\bar{s}\bar{p}}=0
\end{equation*}

we expand 

\begin{eqnarray*}
&=& \sum_{i} \hat{R}\!\left(
E_{k}+E_{\bar{k}},\,
E_{l}+E_{\bar{l}},\,
E_{i}+E_{\bar{\imath}},\,
E_{i}-E_{\bar{\imath}}
\right)
\\
&=& \sum_{i}\left(
\frac{1}{2}C(k,\bar{l},\bar{\imath},i)
-\frac{1}{2}C(k,\bar{l},i,\bar{\imath})
+\frac{1}{2}C(\bar{k},l,\bar{\imath},i)
-\frac{1}{2}C(\bar{k},l,i,\bar{\imath})
\right)
\\
&=& \sum_{i}\left(
C(k,\bar{l},\bar{\imath},i)
-
C(l,\bar{k},\bar{\imath},i)
\right).
\end{eqnarray*}

On the other hand, we can compute Ricci directly: Note that $\left\{
F_{i}\right\} $ is an orthonormal basis for the tangent space. $\ $The set
of tangent vectors $\left\{ \tilde{K}F_{i}\right\} $ is normal to $F_{i}$
but not necessarily orthonormal. \ We can write the trace as follows 
\begin{align*}
\hat{R}ic(F_{k},KF_{l})
&= g^{ii}\hat{R}(F_{i},F_{k},KF_{l},F_{i})
  + \eta^{ij}\hat{R}(\tilde{K}F_{i},F_{k},KF_{l},\tilde{K}F_{j}) \\
&= \sum_{i}\hat{R}(F_{i},F_{k},KF_{l},F_{i})
   - \delta^{ij}\hat{R}(\tilde{K}F_{i},F_{k},KF_{l},\tilde{K}F_{j})
   + \left(\delta^{ij}+\eta^{ij}\right)
     \hat{R}(\tilde{K}F_{i},F_{k},KF_{l},\tilde{K}F_{j}) \\
&= \sum_{i}\hat{R}(F_{i},F_{k},KF_{l},F_{i})
   - \delta^{ij}\hat{R}(KF_{i},F_{k},KF_{l},KF_{j})
   + \omega \ast
\end{align*}
using the fact that 
\begin{equation*}
K-\tilde{K}=\omega \ast G
\end{equation*}
and 
\begin{equation*}
\left( \delta^{ij}+\eta^{ij}\right) = \omega{\ast}G.
\end{equation*}
Thus
\begin{align*}
\hat{R}ic(F_{k},KF_{l}) - \omega \ast G
&= \sum_{i}\hat{R}(F_{i},F_{k},KF_{l},F_{i})
   - \delta^{ij}\hat{R}(KF_{i},F_{k},KF_{l},KF_{j})
\\
&= \sum_{i}\hat{R}(E_{i}+E_{\bar{\imath}},\,E_{k}+E_{\bar{k}},\,E_{l}-E_{\bar{l}},\,E_{i}+E_{\bar{\imath}}) \\
&\quad - \sum_{i}\hat{R}(E_{i}-E_{\bar{\imath}},\,E_{k}+E_{\bar{k}},\,E_{l}-E_{\bar{l}},\,E_{i}-E_{\bar{\imath}})
\\
&= \frac{1}{2}\sum_{i}\Big(
    C(i,\bar{k},l,\bar{\imath})
  + C(i,\bar{k},-\bar{l},i)
  + C(\bar{\imath},k,l,\bar{\imath})
  + C(\bar{\imath},k,-\bar{l},i)
  \Big) \\
&\quad - \frac{1}{2}\sum_{i}\Big(
    C(i,\bar{k},l,-\bar{\imath})
  + C(i,\bar{k},-\bar{l},i)
  + C(-\bar{\imath},k,l,-\bar{\imath})
  + C(-\bar{\imath},k,-\bar{l},i)
  \Big)
\end{align*}

\begin{equation*}
= \sum_{i} C(i,\bar{k},l,\bar{\imath}) - C(\bar{\imath},k,\bar{l},i)
\end{equation*}
We can check that this is exactly the expression derived above. 
\end{proof}

\begin{claim}
In special coordinates, we have 
\begin{equation*}
d\alpha _{n\nabla \psi ^{\perp }}(F_{k},F_{l})=\sum_{i}\left[ C(k,%
\bar{l},\bar{\imath},i)-C(l,\bar{k},\bar{\imath},i)\right] +\omega \ast G
+D\omega \ast G'.
\end{equation*}
\end{claim}

\begin{proof} We compute
\begin{eqnarray*}
\alpha _{n\nabla \psi ^{\perp }}(F_{l}) &=&\omega (\left( n\nabla \psi
\right) ^{\perp },F_{l}) \\
&=&h(K\left( n\nabla \psi \right) ^{\perp },F_{l}) \\
&=&-h(\left( n\nabla \psi \right) ^{\perp },KF_{l}) \\
&=&-h(\left( n\nabla \psi \right) ^{\perp },\tilde{K}F_{l}+\omega
_{l}^{p}F_{p}) \\
&=&-\tilde{K}F_{l}n\psi  \\
&=&-\left( KF_{l}-\omega _{l}^{p}F_{p}\right) n\psi  \\
&=&-\left( F_{l}-2T_{l}^{\bar{s}}E_{\bar{s}}-\omega _{l}^{p}F_{p}\right)
n\psi  \\
&=&-F_{l}n\psi +2T_{l}^{s}E_{\bar{s}}n\psi +\omega _{l}^{p}F_{p}n\psi 
\end{eqnarray*}%
so%
\begin{eqnarray*}
d\alpha _{n\left( \nabla \psi \right) ^{\perp }}(F_{k},F_{l}) &=&F_{k}\omega
(\left( n\nabla \psi \right) ^{\perp },F_{l})-F_{l}\omega (K\left( n\nabla
\psi \right) ^{\perp },F_{k}) \\
&=&-\left( F_{k}F_{l}-F_{l}F_{k}\right) n\psi  \\
&&+F_{k}\left(2 T_{l}^{s}E_{\bar{s}}n\psi \right) -F_{l}\left(2 T_{k}^{s}E_{%
\bar{s}}n\psi \right) + \\
&&F_{k}\left( \omega _{l}^{p}F_{p}n\psi \right) -F_{l}\left( \omega
_{k}^{p}F_{p}n\psi \right) 
\end{eqnarray*}%
\begin{eqnarray*}
&=&2\left( T_{lk}^{s}-T_{kl}^{s}\right) E_{\bar{s}}n\psi +2T_{l}^{s}\frac{%
\partial ^{2}\left( n\psi \right) }{\partial x^{k}\partial \bar{x}^{s}}%
+2T_{l}^{s}T_{k}^{p}\frac{\partial ^{2}\left( n\psi \right) }{\partial \bar{x}%
^{p}\partial \bar{x}^{s}}-2\left( T_{k}^{s}\frac{\partial ^{2}\left( n\psi
\right) }{\partial x^{l}\partial \bar{x}^{s}}+T_{k}^{s}T_{l}^{p}\frac{%
\partial ^{2}\left( n\psi \right) }{\partial \bar{x}^{p}\partial \bar{x}^{s}}%
\right)  \\
&&+\left( \omega _{l}^{p}F_{k}F_{p}-\omega _{k}^{p}F_{l}F_{p}\right) n\psi
+\left( F_{k}\omega _{l}^{p}-F_{l}\omega _{k}^{p}\right) F_{p}n\psi 
\end{eqnarray*}%
\begin{equation*}
=2\delta _{l}^{s}\frac{\partial ^{2}n\psi }{\partial x^{k}\partial \bar{x}^{s}%
}-2\delta _{k}^{s}\frac{\partial ^{2}n\psi }{\partial x^{l}\partial \bar{x}%
^{s}}+\omega \ast G +D\omega \ast G' .
\end{equation*}%
Now recall that 
\begin{equation*}
n\psi =\frac{1}{2}\left( \ln \rho +\ln \bar{\rho}-\ln \det \left( -c_{i\bar{s%
}}\right) \right) 
\end{equation*}%
so \bigskip 
\begin{equation*}
\frac{\partial \left( n\psi \right) }{\partial \bar{x}^{s}}=\frac{1}{2}\frac{%
\bar{\rho}_{s}}{\bar{\rho}}-\frac{1}{2}(c^{a\bar{r}}c_{a\bar{r}\bar{s}})
\end{equation*}%
\begin{equation*}
\frac{\partial ^{2}\left( n\psi \right) }{\partial x^{k}\partial \bar{x}^{s}}%
=-\frac{1}{2}(c^{a\bar{r}}c_{a\bar{r}\bar{s}k}-c^{a\bar{m}}c^{b\bar{r}}c_{b%
\bar{m}k}c_{a\bar{r}\bar{s}})
\end{equation*}%
and in our special coodindates, get 
\begin{eqnarray*}
d\alpha _{n\left( \nabla \psi \right) ^{\perp }}(F_{k},F_{l})-\omega \ast G
-D\omega \ast G'  &=&2\delta _{l}^{s}\left[ -\frac{1}{2}(c^{a\bar{r}}c_{a\bar{r}%
\bar{s}k})\right] -2\delta _{k}^{s}\left[ -\frac{1}{2}(c^{a\bar{r}}c_{a\bar{r}%
\bar{s}l})\right]  \\
&=&\sum_{a}C(a,\bar{a},\bar{l},k)-C(a,\bar{a},\bar{k},l) \\
&=&\sum_{i}C(i,\bar{\imath},\bar{l},k)-C(i,\bar{\imath},\bar{k},l)
\\
&=&\sum_{i}C(k,\bar{l},\bar{\imath},i)-C(l,\bar{k},\bar{\imath},i).
\end{eqnarray*}
\end{proof}

\begin{proof}[Proof of Proposition \ref{evl4}.]
    The proof follows as a consequence of the above lemmas. 
\end{proof}

We end this section by noting that on Lagrangian submanifolds, we get the analogue of Dazord's Theorem \cite{dazord1981geometrie}: 

\begin{theorem}
Along $L$ \label{Dazord}, we get
\begin{equation}
d\left(  F^{\ast}\alpha_{\vec{H}}\right)  =F^{\ast}{\hat{R}ic}(K\cdot,\cdot)).
\end{equation}

\end{theorem}

\section{Long-time Existence and Regularity}

In this section, we solve the more interesting problem of long-time existence. The
Ma-Trudinger-Wang condition (A3) was originally introduced while developing maximum principle methods for Monge-Amp\`ere equation \cite{MTW}.  It was given a geometric formulation by Kim-McCann \cite{KM} to be equivalent to the condition 
\begin{equation}
\label{MTW}\hat{R}\big(\xi\oplus\bar{0},0 \oplus\bar{\xi},\xi\oplus\bar{0},0
\oplus\bar{\xi}\big) > 0
\end{equation}
for all points $(x,\bar{x}) \in M\times\bar{M}$ and all nonvanishing tangent
vectors $\xi\in T_{x} M$ and $\bar{\xi} \in T_{\bar{x}} \bar{M}$ satisfying
$h(\xi\oplus\bar{0},0 \oplus\bar{\xi})=0$. Here, $0 \in T_{x} M$ and $\bar{0}
\in T_{\bar{x}} \bar{M}$ denote zero vectors (see \cite[Remark 4.2]{KMW}).
Note that condition \eqref{MTW} is not affected by a conformal change of
the pseudo-metric $h$.

\medskip
We now restate our main theorem in a more precise form.

\begin{theorem}
\label{thm:main_reg}Suppose that $L$ is an immersed compact spacelike Lagrangian
submanifold of a GKM. Suppose that the
GKM satisfies a positive cross-curvature condition on a compact set avoiding the cut locus. Then, if the generalized Lagrangian mean curvature flow stays in this compact set away from the cut-locus along
the flow, then the flow exists for all time; uniform estimates are preserved, and
the flow converges exponentially to a stationary submanifold.
\end{theorem}

In the sequel, we assume that $M$ and $\bar{M}$ are compact, that
$L$ is compactly immersed, and the immersion lies in a compact region avoiding the cut locus. By Lemma \ref{graphical} and the discussion following it, we may assume that locally the immersion can be represented as the graph over $M$ of a function $T$, which in turn can locally be written as a cost-exponential. For many geometric computations, we therefore work in such neighborhoods, viewing the immersion locally as a map from $M$.

Double-dipping notation, as before, we write
\begin{align*}
\rho(x)  &  =\frac{d\rho}{dx^{1}\wedge...\wedge dx^{n}}\\
\bar{\rho}(x)  &  =\frac{d\bar{\rho}}{d\bar{x}^{1}\wedge...\wedge d\bar{x}^{n}}.
\end{align*}
Recalling Proposition \ref{theta is defined}
and  (\ref{theta defined here})
we define
\begin{equation}
\theta =-\frac{1}{2} \left(\ln\det g_{ij} - \ln\rho + \ln\bar{\rho} - \ln\det\left( b_{is} \right) \right)
\label{def:theta}
\end{equation} taking $g_{ij}$ to be the induced metric on the graph. 
as shown in Proposition \ref{theta is defined}.
This function is geometric, i.e, this quantity is well-defined regardless of
the coordinate system we use.

First, observe that this quantity satisfies a maximum principle. To see this, we do the following geometric computation.
\begin{claim}
\label{thetalution}
\[
\frac{d}{dt}\theta(x,t) = \Delta_g \theta(x,t) 
   + n\, g(\nabla\psi(x),\nabla\theta(x,t)).
\]
\end{claim}

\begin{proof}
To be clear, this is computed along a normal flow where the location $x(t)$ is moving along a normal trajectory (not a vertical one). Recall that along $L$,
\[
\Omega|_L = e^{\mathbf{k}\theta(x,t) + n\psi(x)} \, dVol_L .
\]
Thus,
\begin{align*}
\frac{d}{dt}\big(\Omega|_L\big)
    &= F^* \mathcal{L}_{\vec{H}-n(\nabla\psi)^\perp}\Omega  \\
    &= F^* d\big( \iota_{\vec{H}-n(\nabla\psi)^\perp}\Omega \big) \\
    &= F^* d(\iota_{K\nabla\theta} \Omega)  \\
    &= d\, F^*(\iota_{K\nabla\theta}\Omega)  \\
    &= d\, F^*(\mathbf{k}\cdot \iota_{\nabla\theta}\Omega),
\end{align*}
using that $\Omega$ is para-holomorphic.

Hence
\[
\frac{d}{dt}\big(\Omega|_L\big)
    = d\!\left( \mathbf{k}\cdot \iota_{\nabla\theta}
        \big(e^{\mathbf{k}\theta+n\psi} dVol_L\big) \right).
\]

Since
\[
\iota_{\nabla\theta} dVol_L = *\, d\theta ,
\]
we obtain
\begin{align*}
\frac{d}{dt}\big(\Omega|_L\big)
    &= d\!\left( \mathbf{k}\, e^{\mathbf{k}\theta+n\psi} *d\theta \right) \\
    &= \mathbf{k}\left( (\mathbf{k}\, d\theta + n\, d\psi)\,
         e^{\mathbf{k}\theta+n\psi} \wedge *d\theta
         + e^{\mathbf{k}\theta+n\psi} d(*d\theta) \right) \\
    &= e^{\mathbf{k}\theta+n\psi}
       \left[ d\theta\wedge *d\theta
              + \mathbf{k}\left( n\, d\psi\wedge *d\theta + d(*d\theta) \right)
       \right].
\end{align*}

On the other hand,
\[
\frac{d}{dt}\left(e^{\mathbf{k}\theta+n\psi} dVol_L\right)
    = \left(\mathbf{k}\,\frac{d}{dt}\theta(x,t)
            + n\,\frac{d}{dt}\psi(F(x,t))\right)
        e^{\mathbf{k}\theta+n\psi} dVol_L
        + e^{\mathbf{k}\theta+n\psi} \frac{d}{dt} dVol_L.
\]

Matching real and imaginary parts, we obtain
\[
\frac{d}{dt}\theta(x,t)\, dVol_L
    = n\, d\psi\wedge *d\theta + d(*d\theta),
\]
and
\[
n\,\frac{d}{dt}\psi(F(x,t)) + \frac{d}{dt} dVol_L
    = d\theta\wedge *d\theta.
\]

Therefore,
\begin{align*}
\frac{d}{dt}\theta(x,t)
    &= *^{-1} d(*d\theta) + *^{-1}\big(n\, d\psi\wedge *d\theta\big) \\
    &= - d^* d\theta + n\, g(\nabla\psi,\nabla\theta) \\
    &= \Delta_g \theta + n\, g(\nabla\psi,\nabla\theta).
\end{align*}

Finally,
\[
\frac{d}{dt} dVol_L 
    = |\nabla\theta|^2 - n\,\frac{d}{dt}\psi(F(x,t)).
\]

\end{proof}

It follows from Claim \ref{thetalution} that $\theta$ must satisfy a maximum principle. 

\begin{corollary} \label{thetamaxprinciple}
    Given a generalized Lagrangian mean curvature flow, the value of $\theta$ is bounded above and below by its initial values. 
\end{corollary}

\subsection{Slope Estimates for the Graph} 

To make a global sense of slope bounds, we follow the approach of \cite{BLMR} and introduce a global Riemannian metric 
$\hat{S}$ on the product space. In this framework, the slope can be measured by computing the 
$\hat{S}$-length of tangent vectors to the immersion that have unit length with respect to the pseudo-Riemannian metric.

\subsubsection{Construction of the Auxiliary Metric $\hat{S}$}

We use the indefinite metric to define a fully Riemannian metric on 
$M \times \bar{M}$.  Pick an arbitrary Riemannian metric $m$ on $M$.  
For vectors $\bar{W}_i \in T_{\bar{x}}\bar{M}$ define
\[
\bar{m}(\bar{W}_{1},\bar{W}_{2})
    = 4\, m^{ij}\, h(V_{i},\bar{W}_{1})\, h(V_{j},\bar{W}_{2}),
\]
where the trace is taken over some basis $\{V_i\}$ for $T_{p}M$.  
We can check that $\bar{m}$ defines a positive definite symmetric bilinear form 
on the tangent space of $\bar{M}$ at any point in the product where $h$ 
is defined.

Now define the auxiliary metric
\[
\hat{S} = m + \bar{m},
\]
which is a positive definite metric on the product space.  
Note that this is not a product metric, as the second factor depends 
on the first point.

Our goal is to bound the ratio
\begin{equation}
\mathcal{R}(V) = \frac{\hat{S}(V,V)}{h(V,V)}
\label{def:ratio}
\end{equation}
over all vectors $V$ in the tangent space to the graph.

Note that this quotient can be maximized for two different geometric reasons.  
One is that the tangent vector is very flat, in which case $h(V,V)$ is small 
relative to $\hat{S}(V,V)$.  
The other is that the tangent vector is very steep.

\subsubsection{Coordinates at the Maximum Point}

Since $L$ is compact, the quantity defined by (\ref{def:ratio}) is bounded whenever the immersion is spacelike, in which case it represents the ratio of two Riemannian metrics on the same tangent space. Consequently, (\ref{def:ratio}) attains a maximum on the unit tangent bundle. Let $\left( x,\bar{x}\right) $ be a point where this maximum is achieved.
We now describe the tangent space at this point more precisely.

At this point, take normal coordinates with respect to $m$.  Then choose Type I special coordinates for $\bar{M}$.
We have a product coordinate chart with%
\begin{equation*}
h=\frac{1}{2}\left( 
\begin{array}{cc}
0 & \delta _{i\bar{s}} \\ 
\delta _{\bar{s}i} & 0%
\end{array}%
\right) 
\end{equation*}%
at the point $\left( x,\bar{x}\right) $. We also know that the first chart
was chosen normal by $m.$ Note that all derivatives in the chart are
determined by normal coordinates in the original metric, linearly modifying
by the inverse of $c_{is}$ which is controlled on a compact region. 
In this choice of coordinates, at the point, we have 
\begin{align*}
\hat{S}\left( E_{i},E_{j}\right) & =\delta _{ij} \\
\hat{S}\left( E_{\bar{k}},E_{\bar{l}}\right) & =4\delta ^{ij}\frac{1}{2}\delta
_{ki}\frac{1}{2}\delta _{lj}=\delta _{kl}.
\end{align*}%
We  know that the K\"{a}hler form at the point is still 
\begin{equation*}
\omega (\cdot ,\cdot )=h(K\cdot ,\cdot )
\end{equation*}%
so a graphical tangent vector over $M$ of the form 
\begin{equation*}
\widetilde{F}_{i}=E_{i}+T_{i}^{\bar{s}}E_{\bar{s}}
\end{equation*}%
must satisfy $\omega (\widetilde{F}_{i},\widetilde{F}_{j})=0$ which, using $c_{i\bar{s}}=-\delta _{i%
\bar{s}}$ yields 
\begin{equation*}
T_{i}^{\bar{j}}=T_{j}^{\bar{\imath}}
\end{equation*}%
at the point $x.$ \ In particular the matrix $DT$ is symmetric, and we may
diagonalize 
\begin{equation*}
DT=\left( \lambda _{1},...,\lambda _{n}\right) 
\end{equation*}%
where we are choosing the order of the eigenvalues $\lambda _{1}\geq ...\geq
\lambda _{n}$ and noting that the space-like condition requires the
eigenvalues of $DT$ to be positive. \ Thus at a point we have 
\begin{equation*}
\widetilde{F}_{i}=E_{i}+\lambda _{i}E_{\bar{\imath}}
\end{equation*}%
and  
\begin{equation*}
\frac{\hat{S}(\widetilde{F}_{i},\widetilde{F}_{i})}{h(\widetilde{F}_{i},\widetilde{F}_{i})}
=\frac{1+\lambda _{i}^{2}}{\lambda _{i}}.
\end{equation*}%
Next we claim that the maximum of the ratio must occur at either $\lambda
_{1}$ or $\lambda _{n}.$

First, we note that the functiom 
\begin{equation*}
\lambda \rightarrow \frac{1+\lambda ^{2}}{\lambda }
\end{equation*}%
is convex so the maximum over a set of finite values is either at the largest or
the smallest. \ So either 
\begin{equation*}
\frac{1+\lambda _{1}^{2}}{\lambda _{1}}\geq \frac{1+\lambda _{i}^{2}}{%
\lambda _{i}}
\end{equation*}%
or 
\begin{equation*}
\frac{1+\lambda _{n}^{2}}{\lambda _{n}}\geq \frac{1+\lambda _{i}^{2}}{%
\lambda _{i}}.
\end{equation*}
A basic linear algebra argument gaurantees that the quotient will be maximized along one of the these two eigenvectors.  
So at the point where the maximum occurs, we have chosen coordinates $
\widetilde{F}_{i}=E_{i}+\lambda _{i}E_{\bar{\imath}}$ and the maximum ratio $\mathcal{R}%
(V)$ is either at $\widetilde{F}_{1}$ or $\widetilde{F}_{n}$

Rescale the tangent vectors diagonally by
$\widetilde{F}_i \mapsto \frac{1}{\sqrt{\lambda_i}}\,\widetilde{F}_i$. 
The resulting vectors form an orthonormal basis with respect to the induced metric $g$. We then exponentiate these vectors with respect to $g$ in a neighborhood, obtaining a new coordinate system on $M$ that agrees with the original coordinates at the point 
$x$ and is normal there. The purpose of this construction is to allow us to carry out computations in normal coordinates at the point.

Define the set of tangent vectors $F_i$, to be the coordinate tangent frame obtained by exponentiating as
$$ \frac{1}{\sqrt{\lambda_i}} E_i$$ with respect to the induced metric. Note that at this point, these are represented in the ambient (original coordinates) as 
$$ \frac{1}{\sqrt{\lambda_i}} E_i + \sqrt{\lambda_i} E_{\bar{i}}.$$


\subsubsection{ Maximum Principle Formulas for the Slope Ratio}

The following is a general statement concerning the maximum of ratios for tensors on a Riemannian manifold; it does not depend on any specific properties of the manifold or submanifolds considered here.

\begin{proposition}
\label{p1}
Let $S$ be a symmetric $(0,2)$-tensor on $M$. 
Suppose $p_0$ is a point where the function $R(V)$ achieves its maximum, 
and let $F_{\maxidx}$ be a maximizing vector. 
If $S$ is diagonal at $p_0$ with respect to normal coordinates 
$\{F_1,\ldots,F_n\}$ containing $F_{\maxidx}$, then
\[
\nabla_{F_i}\nabla_{F_i} S(F_{\maxidx},F_{\maxidx}) \le 0
\quad \text{for all } i .
\]
\end{proposition}

\begin{proof}
Working entirely intrinsically, we get 
\begin{align*}
\nabla\nabla S(F_{\maxidx},F_{\maxidx},F_{i},F_{i})
&= F_{i}\nabla S(F_{\maxidx},F_{\maxidx},F_{i})\\
&\quad -\nabla S(\nabla_{F_{i}}F_{\maxidx},F_{\maxidx},F_{i})
       -\nabla S(F_{\maxidx},\nabla_{F_{i}}F_{\maxidx},F_{i})
       -\nabla S(F_{\maxidx},F_{\maxidx},\nabla_{F_{i}}F_{i}).
\end{align*}

Evaluating this at the origin (note all intrinsic connection terms vanish), we get
\begin{align*}
\nabla\nabla S(F_{\maxidx},F_{\maxidx},F_{i},F_{i})
&= F_{i}\nabla S(F_{\maxidx},F_{\maxidx},F_{i})\\
&= F_{i}\left(F_{i}S(F_{\maxidx},F_{\maxidx})
     - S(\nabla_{F_{i}}F_{\maxidx},F_{\maxidx})
     - S(F_{\maxidx},\nabla_{F_{i}}F_{\maxidx})\right) \\
&= S_{\maxidx\maxidx,ii}
  - 2\big[\nabla S(\nabla_{F_{i}}F_{\maxidx},F_{\maxidx},F_{i})
         + S(\nabla_{F_{i}}\nabla_{F_{i}}F_{\maxidx},F_{\maxidx})
         + S(\nabla_{F_{i}}F_{\maxidx},\nabla_{F_{i}}F_{\maxidx})\big]\\
&= S_{\maxidx\maxidx,ii}
  - 2S(\nabla_{F_{i}}\nabla_{F_{i}}F_{\maxidx},F_{\maxidx}).
\end{align*}

Let
\[
v = \left[(\nabla_{F_{i}}\nabla_{F_{i}}F_{\maxidx})\cdot F_{\maxidx}\right] F_{\maxidx}
    = \beta F_{\maxidx}.
\]

Since $S$ is diagonal,
\begin{align*}
\nabla\nabla S(F_{\maxidx},F_{\maxidx},F_{i},F_{i})
&= S_{\maxidx\maxidx,ii} - 2S(\beta F_{\maxidx},F_{\maxidx})\\
&= S_{\maxidx\maxidx,ii} - 2\beta S_{\maxidx\maxidx}.
\end{align*}

Next compute derivatives of
\[
\frac{S(F_{\maxidx},F_{\maxidx})}{g(F_{\maxidx},F_{\maxidx})}.
\]

The first derivative is:
\[
F_{i}\left( \frac{S(F_{\maxidx},F_{\maxidx})}{g(F_{\maxidx},F_{\maxidx})} \right)
= \frac{S_{\maxidx\maxidx,i}g_{\maxidx\maxidx}
      - g_{\maxidx\maxidx,i}S_{\maxidx\maxidx}}
       {g_{\maxidx\maxidx}^{2}}.
\]

The second derivative is:
\begin{align*}
F_{i}F_{i} \left( \frac{S(F_{\maxidx},F_{\maxidx})}{g(F_{\maxidx},F_{\maxidx})}\right)
= \frac{
 S_{\maxidx\maxidx,ii}g_{\maxidx\maxidx}
 + S_{\maxidx\maxidx,i}g_{\maxidx\maxidx,i}
 - g_{\maxidx\maxidx,ii}S_{\maxidx\maxidx}
 - g_{\maxidx\maxidx,i}S_{\maxidx\maxidx,i}}
 {g_{\maxidx\maxidx}^{2}}
\\
 - \frac{2(S_{\maxidx\maxidx,i}g_{\maxidx\maxidx}
   - g_{\maxidx\maxidx,i}S_{\maxidx\maxidx})}{g_{\maxidx\maxidx}^{3}}
   g_{\maxidx\maxidx,i}.
\end{align*}

At the origin, this simplifies to
\[
F_{i}F_{i}\frac{S(F_{\maxidx},F_{\maxidx})}{g(F_{\maxidx},F_{\maxidx})}
= S_{\maxidx\maxidx,ii} - g_{\maxidx\maxidx,ii} S_{\maxidx\maxidx}.
\]

Next compute
\begin{align*}
g_{\maxidx\maxidx,ii}
&= F_{i}F_{i} g(F_{\maxidx},F_{\maxidx})\\
&= F_{i}\, 2g(\nabla_{F_{i}}F_{\maxidx},F_{\maxidx})\\
&= 2g(\nabla_{F_{i}}\nabla_{F_{i}}F_{\maxidx},F_{\maxidx})
  +2g(\nabla_{F_{i}}F_{\maxidx},\nabla_{F_{i}}F_{\maxidx})\\
&= 2g(\nabla_{F_{i}}\nabla_{F_{i}}F_{\maxidx},F_{\maxidx})
   \quad (\text{at origin})\\
&= 2g(\beta F_{\maxidx},F_{\maxidx})
 = 2\beta.
\end{align*}

So we get
\begin{align*}
0
&\ge F_{i}F_{i}\frac{S(F_{\maxidx},F_{\maxidx})}{g(F_{\maxidx},F_{\maxidx})} \\
&= S_{\maxidx\maxidx,ii} - 2\beta S_{\maxidx\maxidx}\\
&= \nabla\nabla S(F_{\maxidx},F_{\maxidx},F_{i},F_{i}).
\end{align*}

Thus, the tensorial statement is
\[
\nabla_{F_i}\nabla_{F_i}S(F_{\maxidx},F_{\maxidx})\leq0.
\]
\end{proof}

Next, we state the following proposition, which is based on \cite[Proposition 1]{BLMR}.

\begin{proposition}
Let $L$ be a Lagrangian submanifold and let $\hat S$ be a tensor on the ambient manifold. 
Denote by $S=\hat S|_L$ its restriction to $L$. Suppose that at a point $p\in L$ the induced metric $g$ is diagonalized by a frame 
$\{F_1,\dots,F_n\}$. Then at $p$ we have
\begin{align*}
g^{mm}\nabla_{F_m}\nabla_{F_m}S(F_{\maxidx},F_{\maxidx})
&= g^{mm}\hat\nabla_{F_m}\hat\nabla_{F_m}\hat S(F_{\maxidx},F_{\maxidx})
   + \hat\nabla_{\vec H}\hat S(F_{\maxidx},F_{\maxidx}) \\
&\quad + 4g^{mm}\hat\nabla_{F_m}\hat S(A(F_m,F_{\maxidx}),F_{\maxidx}) \\
&\quad + 2g^{mm}\hat S(A(F_m,F_{\maxidx}),A(F_m,F_{\maxidx})) \\
&\quad + 2g^{mm}\hat S\!\left((\hat\nabla_{F_{\maxidx}}A)^{\perp}(F_m,F_m),F_{\maxidx}\right) \\
&\quad + 2g^{mm}g^{kl}\hat R(F_m,F_{\maxidx},F_m,KF_k)\,\hat S(KF_l,F_{\maxidx}) \\
&\quad - 2g^{mm}A(F_m,F_{\maxidx})\cdot A(F_m,F_j)g^{ij}\,\hat S(F_i,F_{\maxidx}),
\end{align*}
where $\{KF_l\}$ denotes a basis for the normal space.
\end{proposition}
\begin{proof}
The proof follows from the argument used in Proposition 1 of \cite{BLMR}. For
the convenience of the reader, we include the details here.

Take tangent vectors $\{V,X,Y,W\}$ (assume these are in normal coordinates at
a point with respect to the induced metric):
\begin{align*}
\nabla_{V}S(X,Y) &  =VS(X,Y)-S(\nabla_{V}X,Y)-S(X,\nabla_{V}Y)\\
\hat{\nabla}_{V}\hat{S}(X,Y) &  =V\hat{S}(X,Y)-\hat{S}(\hat{\nabla}%
_{V}X,Y)-S(X,\hat{\nabla}_{V}Y).
\end{align*}
Then
\begin{align*}
\nabla_{V}S(X,Y) &  =\hat{\nabla}_{V}\hat{S}(X,Y)+\hat{S}(\hat{\nabla}%
_{V}X,Y)+S(X,\hat{\nabla}_{V}Y)-S(\nabla_{V}X,Y)-S(X,\nabla_{V}Y)\\
&  =\hat{\nabla}_{V}\hat{S}(X,Y)+\hat{S}(A(V,X),Y)+S(X,A(V,Y)).
\end{align*}
and
\begin{align*}
\nabla_{W}\nabla_{V}S(X,Y) &  =W\nabla_{V}S(X,Y)-\nabla_{\nabla_{W}%
V}S(X,Y)-\nabla_{V}S(\nabla_{W}X,Y)-\nabla_{V}S(X,\nabla_{W}Y)\\
&  =W\left[  \hat{\nabla}_{V}\hat{S}(X,Y)+\hat{S}%
(A(V,X),Y)+S(X,A(V,Y))\right]  \\
&  \qquad-\nabla_{\nabla_{W}V}S(X,Y)-\nabla_{V}S(\nabla_{W}X,Y)-\nabla
_{V}S(X,\nabla_{W}Y)\\
\hat{\nabla}_{W}\hat{\nabla}_{V}\hat{S}(X,Y) &  =W\hat{\nabla}_{V}\hat
{S}(X,Y)-\hat{\nabla}_{\hat{\nabla}_{W}V}\hat{S}(X,Y)-\hat{\nabla}_{V}\hat
{S}(\hat{\nabla}_{W}X,Y)-\hat{\nabla}_{V}\hat{S}(X,\hat{\nabla}_{W}Y).
\end{align*}
So
\begin{align*}
\nabla_{W}\nabla_{V}S(X,Y) &  =\hat{\nabla}_{W}\hat{\nabla}_{V}\hat
{S}(X,Y)+\hat{\nabla}_{\hat{\nabla}_{W}V}\hat{S}(X,Y)+\hat{\nabla}_{V}\hat
{S}(\hat{\nabla}_{W}X,Y)+\hat{\nabla}_{V}\hat{S}(X,\hat{\nabla}_{W}Y)\\
&  \qquad+W\left[  \hat{S}(A(V,X),Y)+S(X,A(V,Y))\right]  \\
&  \qquad-\nabla_{\nabla_{W}V}S(X,Y)-\nabla_{V}S(\nabla_{W}X,Y)-\nabla
_{V}S(X,\nabla_{W}Y)\\
&  =\hat{\nabla}_{W}\hat{\nabla}_{V}\hat{S}(X,Y)\\
&  \quad+\left[  \hat{\nabla}_{\hat{\nabla}_{W}V}\hat{S}-\nabla_{\nabla_{W}%
V}S\right]  (X,Y)\\
&  \quad+\hat{\nabla}_{V}\hat{S}(\hat{\nabla}_{W}X,Y)-\nabla_{V}S(\nabla
_{W}X,Y)\\
&  \quad+\hat{\nabla}_{V}\hat{S}(X,\hat{\nabla}_{W}Y)-\nabla_{V}S(X,\nabla
_{W}Y)\\
&  \quad+W\left[  \hat{S}(A(V,X),Y)+\hat{S}(X,A(V,Y))\right]  .
\end{align*}
Since $\hat{\nabla}_{W}V=A(W,V)$, this becomes
\begin{align*}
\nabla_{W}\nabla_{V}S(X,Y) &  =\hat{\nabla}_{W}\hat{\nabla}_{V}\hat
{S}(X,Y)+\hat{\nabla}_{A(W,V)}\hat{S}(X,Y)\\
&  \quad+\hat{\nabla}_{V}\hat{S}(A(W,X),Y)+\hat{\nabla}_{V}\hat{S}(X,A(W,Y))\\
&  \quad+W\left[  \hat{S}(A(V,X),Y)+\hat{S}(X,A(V,Y))\right]  .
\end{align*}
Expanding the last term:
\[
W\hat{S}(A(V,X),Y)=\hat{\nabla}_{W}\hat{S}(A(V,X),Y)+\hat{S}\left(  \left(
\hat{\nabla}_{W}A\right)  (V,X),Y\right)  +\hat{S}(A(V,X),\hat{\nabla}_{W}Y).
\]
Now
\[
\left(  \hat{\nabla}_{W}A\right)  (V,X)=\nabla_{W}^{\perp}A(V,X)+g^{ij}%
h\left(  (\hat{\nabla}_{W}A)(V,X),F_{j}\right)  F_{i}%
\]
and by the Codazzi equation
\[
\nabla_{W}^{\perp}A(V,X)=\nabla_{X}^{\perp}A(W,V)+g^{kl}\hat{R}(W,X,V,KF_{k}%
)KF_{l}%
\]
(here using the fact the $K$ is a sign inverting isometry to the normal space,
giving a positive sign)\ 

Also note that
\[
h\left(  (\hat{\nabla}_{W}A)(V,X),F_{j}\right)  +h\left(  A(V,X),\hat{\nabla
}_{W}F_{j}\right)  =0
\]
Thus
\begin{align*}
W\hat{S}(A(V,X),Y)  & =\hat{\nabla}_{W}\hat{S}(A(V,X),Y)\\
& +\hat{S}\left(  \nabla_{X}^{\perp}A(W,V)+g^{kl}\hat{R}(W,X,V,KF_{k}%
)KF_{l},Y\right)  -g^{ij}\hat{S}\left(  h\left(  A(V,X),\hat{\nabla}_{W}%
F_{j}\right)  F_{i},Y\right)  \\
& +\hat{S}(A(V,X),\hat{\nabla}_{W}Y).
\end{align*}

and we have
\begin{align*}
\nabla_{W}\nabla_{V}S(X,Y) &  =\hat{\nabla}_{W}\hat{\nabla}_{V}\hat
{S}(X,Y)+\hat{\nabla}_{A(W,V)}\hat{S}(X,Y)\\
&  \quad+\hat{\nabla}_{V}\hat{S}(A(W,X),Y)+\hat{\nabla}_{V}\hat{S}(X,A(W,Y))\\
&  \quad+\hat{\nabla}_{W}\hat{S}(A(V,X),Y)+\hat{\nabla}_{W}\hat{S}(X,A(V,Y))\\
&  \quad+\hat{S}(A(V,X),A(W,Y))+\hat{S}(A(W,X),A(V,Y))\\
&  \quad+\hat{S}\left(  \nabla_{X}^{\perp}A(W,V)+g^{kl}\hat{R}(W,X,V,KF_{k}%
)KF_{l},Y\right)  \\
&  \qquad\qquad-g^{ij}\hat{S}\left(  h\left(  A(V,X),A(W,F_{j})\right)
F_{i},Y\right)  \\
&  \quad+\hat{S}\!\Big(X,\left(  \hat{\nabla}_{Y}A\right)  ^{\perp}%
(W,V)+\sum\hat{R}(W,Y,V,KF_{l})KF_{l}\\
&  \qquad\qquad-g^{ij}\hat{S}\left(  X,h\left(  A(V,Y),A(W,F_{j})\right)
F_{i}\right)
\end{align*}

Tracing over $V=W=F_{m}$, we obtain
\begin{align*}
g^{mm}\nabla_{F_{m}}\nabla_{F_{m}}S(X,Y) &  =g^{mm}\hat{\nabla}_{F_{m}}%
\hat{\nabla}_{F_{m}}\hat{S}(X,Y)+g^{mm}\hat{\nabla}_{A(F_{m},F_{m})}\hat
{S}(X,Y)\\
&  \quad+g^{mm}\hat{\nabla}_{F_{m}}\hat{S}(A(F_{m},X),Y)+g^{mm}\hat{\nabla
}_{F_{m}}\hat{S}(X,A(F_{m},Y))\\
&  \quad+g^{mm}\hat{\nabla}_{F_{m}}\hat{S}(A(F_{m},X),Y)+g^{mm}\hat{\nabla
}_{F_{m}}\hat{S}(X,A(F_{m},Y))\\
&  \quad+g^{mm}2\hat{S}(A(F_{m},X),A(F_{m},Y))\\
&  \quad+g^{mm}\hat{S}\!\Big(\left(  \hat{\nabla}_{X}A\right)  ^{\perp}%
(F_{m},F_{m})+g^{kl}\hat{R}(F_{m},X,F_{m},KF_{k})KF_{l}\\
&  \qquad\qquad-A(F_{m},X)\cdot A(F_{m},F_{j})g^{ij}F_{i},\,Y\Big)\\
&  \quad+g^{mm}\hat{S}\!\Big(X,\left(  \hat{\nabla}_{Y}A\right)  ^{\perp
}(F_{m},F_{m})+g^{kl}\hat{R}(F_{m},Y,F_{m},KF_{k})KF_{l}\\
&  \qquad\qquad-A(F_{m},Y)\cdot A(F_{m},F_{j})g^{ij}F_{i}\Big).
\end{align*}

Plugging in $X=Y=F_{\maxidx}$ yields the result.

\end{proof}

Next we compute the time derivative.

\begin{proposition}
Along the generalized Lagrangian mean curvature flow,
at a point where $S$ is diagonalized, the following holds: 
\begin{align}
\frac{d}{dt}S_{\maxidx\maxidx}
& - g^{mm}\nabla_{F_{m}}\nabla_{F_{m}}S(F_{\maxidx},F_{\maxidx})  \label{evolution S}\\
=& -n\hat{\nabla}_{\left( \nabla \psi \right) ^{\perp }}S(F_{\maxidx},F_{\maxidx})
   -2n\hat{S}\left( (\hat{\nabla}_{F_{\maxidx}}\left( \nabla \psi \right) ^{\perp })^{T},
     F_{\maxidx}\right)
   -2n\hat{S}\left( \nabla_{F_{\maxidx}}^{\perp }\left( \nabla \psi \right) ^{\perp },
     F_{\maxidx}\right) \nonumber \\
& - g^{mm}\hat{\nabla}_{F_{m}}\hat{\nabla}_{F_{m}}\hat{S}(F_{\maxidx},F_{\maxidx})
  + 2\hat{S}_{\maxidx\maxidx}\left( (\hat{\nabla}_{F_{\maxidx}}\vec{H})^{T}, F_{\maxidx}\right)
  - 4 g^{mm}\hat{\nabla}_{F_{m}}\hat{S}(A(F_{m},F_{\maxidx}),F_{\maxidx}) \nonumber \\
& - 2 g^{mm}\hat{S}(A(F_{m},F_{\maxidx}),A(F_{m},F_{\maxidx}))
  - 2 g^{mm}\sum \hat{R}(F_{m},F_{\maxidx},F_{m},KF_{l})\hat{S}(KF_{l},F_{\maxidx}) \nonumber \\
& + 2 g^{mm}A(F_{m},F_{\maxidx}) \cdot A(F_{m},F_{\maxidx}) \hat{S}_{\maxidx\maxidx}. \nonumber
\end{align}
\end{proposition}

\begin{proof}
First, recall the formula for evolution in the normal direction $V$ 
\begin{equation*}
\frac{d}{dt}S_{ij}
  =\hat{\nabla}_{V}\hat{S}\left( F_{i},F_{j}\right)
  +\hat{S}\left( \hat{\nabla}_{F_{i}}V,F_{j}\right)
  +\hat{S}\left( F_{i},\hat{\nabla}_{F_{j}}V\right).
\end{equation*}

Let $V=\vec{H}-n\left( \nabla \psi \right) ^{\perp }.$ 
Choosing $i,j=\maxidx$ we get 
\begin{align*}
\frac{d}{dt}S_{\maxidx\maxidx}
& =\hat{\nabla}_{V}S(F_{\maxidx},F_{\maxidx})
   +2\hat{S}\left( \hat{\nabla}_{F_{\maxidx}}\left( \vec{H}
   -n\left( \hat{\nabla}\psi \right) ^{\perp }\right),F_{\maxidx}\right) \\[6pt]
&=\hat{\nabla}_{V}S(F_{\maxidx},F_{\maxidx})
  +2\hat{S}\left( (\hat{\nabla}_{F_{\maxidx}}(\vec{H}
  -n(\hat{\nabla}\psi)^{\perp}))^T,F_{\maxidx}\right)
  +2\hat{S}((\hat{\nabla}_{F_{\maxidx}}(\vec{H}
  -n(\nabla\psi)^{\perp}))^\perp ,F_{\maxidx}) \\[6pt]
&=\hat{\nabla}_{V}S(F_{\maxidx},F_{\maxidx})
  -2n\hat{S}((\hat{\nabla}_{F_{\maxidx}}(\hat{\nabla}\psi )^{\perp})^T,F_{\maxidx})
  -2n\hat{S}((\hat{\nabla}_{F_{\maxidx}}(\hat{\nabla}\psi)^{\perp})^\perp,F_{\maxidx}) \\
&\quad +2\hat{S}((\hat{\nabla}_{F_{\maxidx}}H)^T,F_{\maxidx})
      +2\hat{S}((\hat{\nabla}_{F_{\maxidx}}H)^\perp,F_{\maxidx}).
\end{align*}

Next, recall the formula
\begin{align*}
g^{mm}\nabla _{F_{m}}\nabla _{F_{m}}S(F_{\maxidx},F_{\maxidx})
& =g^{mm}\hat{\nabla}_{F_{m}}\hat{\nabla}_{F_{m}}\hat{S}(F_{\maxidx},F_{\maxidx})
  +\hat{\nabla}_{\vec{H}}\hat{S}(F_{\maxidx},F_{\maxidx}) \\
& \quad +4g^{mm}\hat{\nabla}_{F_{m}}\hat{S}(A(F_{m},F_{\maxidx}),F_{\maxidx}) \\
& \quad +2g^{mm}\hat{S}(A(F_{m},F_{\maxidx}),A(F_{m},F_{\maxidx})) \\
& \quad +2g^{mm}\hat{S}\!\left( (\hat{\nabla}_{F_{\maxidx}}A)^{\perp}
        (F_{m},F_{m}),F_{\maxidx}\right)  \\
& \quad +2g^{mm} g^{kl} \hat{R}(F_{m},F_{\maxidx},F_{m},KF_{k})\hat{S}\left(
        KF_{l},F_{\maxidx}\right)  \\
& \quad -2g^{mm}A(F_{m},F_{\maxidx})\cdot A(F_{m},F_{j})g^{ij}
            \hat{S}\left(F_{i},F_{\maxidx}\right).
\end{align*}

Compute the difference:
\begin{align*}
& \frac{d}{dt}S_{\maxidx\maxidx}
 -g^{mm}\nabla _{F_{m}}\nabla _{F_{m}}S(F_{\maxidx},F_{\maxidx}) \\
&= \hat{\nabla}_{V}S(F_{\maxidx},F_{\maxidx})
 -2n\hat{S}((\hat{\nabla}_{F_{\maxidx}}(\hat{\nabla}\psi )^{\perp })^{T},F_{\maxidx})
 -2n\hat{S}((\hat{\nabla}_{F_{\maxidx}}(\hat{\nabla}\psi )^{\perp })^{\perp},F_{\maxidx}) \\
& \quad +2\hat{S}((\hat{\nabla}_{F_{\maxidx}}\vec{H})^{T},F_{\maxidx})
       +2\hat{S}((\hat{\nabla}_{F_{\maxidx}}\vec{H})^{\perp },F_{\maxidx}) \\
& \quad -g^{mm}\hat{\nabla}_{F_{m}}\hat{\nabla}_{F_{m}}\hat{S}(F_{\maxidx},F_{\maxidx})
       -\hat{\nabla}_{H}\hat{S}(F_{\maxidx},F_{\maxidx}) \\
& \quad -4g^{mm}\hat{\nabla}_{F_{m}}\hat{S}(A(F_{m},F_{\maxidx}),F_{\maxidx}) \\
& \quad -2g^{mm}\hat{S}(A(F_{m},F_{\maxidx}),A(F_{m},F_{\maxidx})) \\
& \quad -2g^{mm}\hat{S}\!\left( (\hat{\nabla}_{F_{\maxidx}}A)^{\perp}
        (F_{m},F_{m}),F_{\maxidx}\right)  \\
& \quad -2g^{mm} g^{kl} \hat{R}(F_{m},F_{\maxidx},F_{m},KF_{k})\hat{S}(KF_{l},F_{\maxidx}) \\
& \quad +2g^{mm}A(F_{m},F_{\maxidx})\cdot A(F_{m},F_{\maxidx})\hat{S}_{\maxidx\maxidx}.
\end{align*}

Finally, collecting terms and using that \(S\) is diagonalized, also noting that 

$$ 
g^{mm}\hat{S}\!\left( (\hat{\nabla}_{F_{\maxidx}}A)^{\perp}
        (F_{m},F_{m}),F_{\maxidx}\right) = \hat{S}((\hat{\nabla}_{F_{\maxidx}}\vec{H})^{\perp },F_{\maxidx}) $$
the expression becomes
\begin{align*}
&=-n\hat{\nabla}_{(\nabla \psi )^{\perp }}S(F_{\maxidx},F_{\maxidx})
   -2n\hat{S}((\hat{\nabla}_{F_{\maxidx}}(\nabla \psi )^{\perp })^{T},F_{\maxidx})
   -2n\hat{S}(\nabla _{F_{\maxidx}}^{\perp }(\nabla \psi )^{\perp },F_{\maxidx}) \\
&\quad -g^{mm}\hat{\nabla}_{F_{m}}\hat{\nabla}_{F_{m}}\hat{S}(F_{\maxidx},F_{\maxidx}) \\
&\quad +2\hat{S}_{\maxidx\maxidx}((\hat{\nabla}_{F_{\maxidx}}H)^{T})\cdot F_{\maxidx}
       -4g^{mm}\hat{\nabla}_{F_{m}}\hat{S}(A(F_{m},F_{\maxidx}),F_{\maxidx}) \\
&\quad -2g^{mm}\hat{S}(A(F_{m},F_{\maxidx}),A(F_{m},F_{\maxidx})) \\
&\quad -2g^{mm} g^{kl}\hat{R}(F_{m},F_{\maxidx},F_{m},KF_{k})\hat{S}(KF_{l},F_{\maxidx}) \\
&\quad +2g^{mm}A(F_{m},F_{\maxidx})\cdot A(F_{m},F_{\maxidx})\hat{S}_{\maxidx\maxidx}.
\end{align*}

This completes the proof.
\end{proof}

\subsection{The Maximum Principle Applied to the Slope Ratio}


\bigskip
Suppose that $(x_0,t_0)$, with $x_0\in L$ and $t_0\in(0,T)$, is a point at which the maximum of $\mathcal{R}$ is attained. At this point and time we again choose a frame that diagonalizes the metric and the tensor $S$, so that
\[
h(F_i,F_j)=\delta_{ij},
\]
and
\[
S(F_i,F_j)=\delta_{ij}\left(\frac{1}{\lambda_i}+\lambda_i\right).
\]

Let $\maxidx=1$ or $\maxidx=n$ be chosen so that
\[
S_{\maxidx\maxidx}=\max\left\{\frac{1}{\lambda_1}+\lambda_1,\,
\frac{1}{\lambda_n}+\lambda_n\right\}.
\]
For simplicity, define the function
\[
Q(x,t)=\frac{S(F_{\maxidx},F_{\maxidx},t)}{g(F_{\maxidx},F_{\maxidx},t)}
\]
in a neighborhood of $(x_0,t_0)$. We now state the following key claim.

\begin{claim}
There exist constants $\bar C>0$ and $\kappa>0$ such that the following holds. 
If $S_{\maxidx\maxidx}\ge 10$ and $F_{\maxidx}$ is the maximizer of the ratio $\mathcal{R}(F_{\maxidx})$, then at the maximum point we have
\[
\partial_t Q(F_{\maxidx},F_{\maxidx})
\le 
\bar C\, S_{\maxidx\maxidx}^2
-
\kappa\, S_{\maxidx\maxidx}^{2+\frac{1}{n-1}}.
\]
\end{claim}

\begin{proof}
We start by computing
\[
\partial_{t}Q(F_{\maxidx},F_{\maxidx})=\partial_{t}\left(  \frac{S(F_{\maxidx},F_{\maxidx},t)}%
{g(F_{\maxidx},F_{\maxidx},t)}\right)  \Big|_{x_{0},t_{0}}=S_{\maxidx\maxidx,t}-S_{\maxidx\maxidx}g_{\maxidx\maxidx,t}.
\]
We subtract the {nonpositive} quantity (nonpositive by Proposition \ref{p1}) to obtain
\begin{align*}
\partial_{t}Q(F_{\maxidx},F_{\maxidx})  & \leq\partial_{t}Q(F_{\maxidx},F_{\maxidx})-\sum_{i}%
\nabla_{F_{i}}\nabla_{F_{i}}S(F_{\maxidx},F_{\maxidx})\\
&  =S_{\maxidx\maxidx,t}-\sum_{i}\nabla_{F_{i}}\nabla_{F_{i}}S(F_{\maxidx},F_{\maxidx})-S_{\maxidx\maxidx}%
g_{\maxidx\maxidx,t}.
\end{align*}

Next, we plug in \eqref{evolution S}:
\begin{align*}
\partial_{t}Q(F_{\maxidx},F_{\maxidx})
& \leq -\hat{\nabla}_{n\left( \nabla\psi\right)^{\perp}}
      S(F_{\maxidx},F_{\maxidx})
      -2\hat{S}\left( \nabla_{F_{\maxidx}}^{\perp} n\left(\nabla\psi\right)^{\perp},
      F_{\maxidx}\right)
      -\sum_{i}\hat{\nabla}_{F_{i}}\hat{\nabla}_{F_{i}}
      \hat{S}(F_{\maxidx},F_{\maxidx})\\
& \quad -4\sum_{i}\hat{\nabla}_{F_{i}}
      \hat{S}(A(F_{i},F_{\maxidx}),F_{\maxidx})\\
& \quad +2\hat{S}_{\maxidx\maxidx}\left(
        (\hat{\nabla}_{F_{\maxidx}}\vec{H})^{T}\cdot F_{\maxidx}\right)
      -2\hat{S}\left(
        (\hat{\nabla}_{F_{\maxidx}}n\left( \nabla\psi\right)^{\perp})^{T},
        F_{\maxidx}\right)
      -S_{\maxidx\maxidx} g_{\maxidx\maxidx,t}\\
& \quad -2\sum_{i}\hat{S}(A(F_{i},F_{\maxidx}),A(F_{i},F_{\maxidx}))
      -2\sum_{i,l}\hat{R}(F_{i},F_{\maxidx},F_{i},KF_{l})
         \hat{S}(KF_{l},F_{\maxidx})\\
& \quad +2\sum_{i}
      A(F_{i},F_{\maxidx})\cdot A(F_{i},F_{\maxidx})\hat{S}_{\maxidx\maxidx}.
\end{align*}
We rearrange the above quantities as follows:\begin{align*}
&\quad +2\sum_{i}A(F_{i},F_{\maxidx})\cdot A(F_{i},F_{\maxidx})\hat{S}_{\maxidx\maxidx}
\tag{1}\label{eq:L1}\\
&\quad +2\hat{S}_{\maxidx\maxidx}\left( (\hat{\nabla}_{F_{\maxidx}}\vec{H})^{T}\cdot F_{\maxidx}\right)
   -2\hat{S}\left( (\hat{\nabla}_{F_{\maxidx}}n(\nabla\psi)^{\perp})^{T},F_{\maxidx}\right)
   -S_{\maxidx\maxidx}g_{\maxidx\maxidx,t}
\tag{2}\label{eq:L2}\\
&\quad -4\sum_{i}\hat{\nabla}_{F_{i}}\hat{S}(A(F_{i},F_{\maxidx}),F_{\maxidx})
   \;-\;2\sum_{i}\hat{S}(A(F_{i},F_{\maxidx}),A(F_{i},F_{\maxidx}))
\tag{3}\label{eq:L3}\\
&\quad -\hat{\nabla}_{n(\nabla\psi)^{\perp}}S(F_{\maxidx},F_{\maxidx})
   -\sum_{i}\hat{\nabla}_{F_{i}}\hat{\nabla}_{F_{i}}\hat{S}(F_{\maxidx},F_{\maxidx})
\tag{4}\label{eq:L4}\\
&\quad -2\hat{S}\left( \nabla_{F_{\maxidx}}^{\perp}n(\nabla\psi)^{\perp},F_{\maxidx}\right)
\tag{5}\label{eq:L5}\\
&\quad -2\sum_{i,l}\hat{R}(F_{i},F_{\maxidx},F_{i},KF_{l})\hat{S}(KF_{l},F_{\maxidx})
\tag{6}\label{eq:L6}
\end{align*}

We will deal with these lines in order. \ For lines 1-5 we will show that each
expression is bounded by%
\[
C(1+S_{\maxidx\maxidx}^{2})
\]
for some controlled constant $C.$

\textit{Line 1:}\ \ \ Note that this is the sum of metric pairings of two time-like
vectors, so this is non positive.

\textit{Line 2:} Note that
\begin{align*}
\frac{d}{dt}g_{\maxidx\maxidx} 
&=2h(\hat{\nabla}_{F_{\maxidx}}\left(  \vec{H}-n\left(
\nabla\psi\right)  ^{\perp}\right)  ,F_{\maxidx})\\
&=2(\hat{\nabla}_{F_{\maxidx}}\vec{H})^{T}\cdot F_{\maxidx}
   -2\left(  \hat{\nabla}%
_{F_{\maxidx}}n\left(  \nabla\psi\right)  ^{\perp}\right)  ^{T}\cdot F_{\maxidx}
\end{align*}
and that
\[
\hat{S}\left(  (\hat{\nabla}_{F_{\maxidx}}n\left(  \nabla\psi\right)  ^{\perp}%
)^{T},F_{\maxidx}\right)
=(\hat{\nabla}_{F_{\maxidx}}n\left(  \nabla\psi\right)  ^{\perp})^{T}
\cdot F_{\maxidx}\, S(F_{\maxidx},F_{\maxidx}).
\]

So
\begin{align*}
&\quad 
+2\hat{S}_{\maxidx\maxidx}\left( (\hat{\nabla}_{F_{\maxidx}}\vec{H})^{T}\!\cdot F_{\maxidx}\right)
-2\hat{S}\!\left( (\hat{\nabla}_{F_{\maxidx}} n(\nabla\psi)^{\perp})^{T},\,F_{\maxidx}\right)
-S_{\maxidx\maxidx} g_{\maxidx\maxidx,t}
\\[0.5em]
&= 
2\hat{S}_{\maxidx\maxidx}\left( (\hat{\nabla}_{F_{\maxidx}}\vec{H})^{T}\!\cdot F_{\maxidx}\right)
-2\hat{S}\!\left( (\hat{\nabla}_{F_{\maxidx}} n(\nabla\psi)^{\perp})^{T},\,F_{\maxidx}\right)
\\
&\quad 
-S_{\maxidx\maxidx} 2(\hat{\nabla}_{F_{\maxidx}}\vec{H})^{T}\!\cdot F_{\maxidx}
+S_{\maxidx\maxidx} 2\left( \hat{\nabla}_{F_{\maxidx}} n(\nabla\psi)^{\perp} \right)^{\perp}\!\cdot F_{\maxidx}
\\[0.5em]
&=0.
\end{align*}

\textit{Line 3:} Note that as a fixed tensor on the ambient manifold, we have
\[
\hat{\nabla}_{F_{i}}\hat{S}(X,Y)
\leq C_{1}\left\Vert F_{i}\right\Vert
_{\hat{S}}\left\Vert X\right\Vert _{\hat{S}}\left\Vert Y\right\Vert _{\hat{S}},
\]
so
\[
-4\sum_{i}\hat{\nabla}_{F_{i}}\hat{S}(A(F_{i},F_{\maxidx}),F_{\maxidx})
\leq
2\sum_{i}\hat{S}(A(F_{i},F_{\maxidx}),A(F_{i},F_{\maxidx}))
+2\sum_{i}C_{1}^{2}\hat{S}(F_{i},F_{i})\hat{S}(F_{\maxidx},F_{\maxidx})
\]

hence
\begin{align*}
-4\sum_{i}\hat{\nabla}_{F_{i}}\hat{S}(A(F_{i},F_{\maxidx}),F_{\maxidx}
-2\sum_{i}\hat{S}(A(F_{i},F_{\maxidx}),A(F_{i},F_{\maxidx}))  
&\leq 2\sum_{i}C_{1}^{2}\hat{S}(F_{i},F_{i})\hat{S}(F_{\maxidx},F_{\maxidx})\\
&\leq C S_{\maxidx\maxidx}^{2}.
\end{align*}

\textit{Line 4:}
\[
-\hat{\nabla}_{n\left(  \nabla\psi\right)  ^{\perp}}
S(F_{\maxidx},F_{\maxidx})
-\sum_{i}\hat{\nabla}_{F_{i}}\hat{\nabla}_{F_{i}}\hat{S}(F_{\maxidx},F_{\maxidx}).
\]
Note that the first term involves a projection, so
\[
\left\Vert n\left(  \nabla\psi\right)  ^{\perp}\right\Vert _{\hat{S}}
=\left\Vert
h(n\left(  \nabla\psi\right)  ^{\perp},KF_{l})g^{lm}KF_{m}\right\Vert_{\hat{S}} 
\leq S_{\maxidx\maxidx}.
\]
This shows that the first term is bounded by $S_{\maxidx\maxidx}^{2}.$ 
The second is an ambient 4 tensor, so is bounded by$S_{\maxidx\maxidx}^{2}.$\\

\textit{Line 5:} This is Claim \ref{annoying} below.

We finally claim that the contribution from line 6 is bounded above by
\[
-\kappa S_{\maxidx\maxidx}^{2+\frac{1}{n-1}} + C S_{\maxidx\maxidx}^{2},
\]
for some constant $\kappa>0$, bounded away from zero, 
depending on the cross-curvature condition.

To set this up, recall
\[
m+\bar{m}=\delta_{ij}+\delta_{\bar{\imath}
\bar{j}}
\]
at the point, and the orthonormal basis with respect to $g$ has the
expression
\[
F_i=\frac{1}{\sqrt{\lambda_{i}}}E_{i}
  +\sqrt{\lambda_{i}}E_{\bar{\imath}}.
\]
Next, we observe the term
\[
\sum_{i,l}\hat{R}(F_{i},F_{\maxidx},F_{i},KF_{l})\hat{S}(KF_{l},F_{\maxidx})
\]
and note that
\[
\hat{S}(KF_{l},F_{\maxidx})
=\left(  \frac{1}{\lambda_{\maxidx}}-\lambda_{\maxidx}\right)\delta_{l\maxidx}.
\]
and 
\[
\hat{R}(F_{i},F_{\maxidx},F_{i},KF_{\maxidx})
=-\frac{\lambda_{\maxidx}}{\lambda_{i}}\hat{R}_{i\overline{\maxidx} i\overline{\maxidx}}
+\frac{\lambda_{i}}{\lambda_{\maxidx}}\hat{R}_{\overline{\imath}\maxidx\overline{\imath}\maxidx}.
\]

Thus (see Claim \ref{a67} below) 
\begin{align*}
\sum_{i,l}\hat{R}(F_{i},F_{\maxidx},F_{i},KF_{l})\hat{S}(KF_{l},F_{\maxidx})
&=\sum_{i}\left(
  -\frac{\lambda_{\maxidx}}{\lambda_{i}}\hat{R}_{i\overline{\maxidx}i\overline{\maxidx}}
  +\frac{\lambda_{i}}{\lambda_{\maxidx}}\hat{R}_{\overline{\imath}\maxidx\overline{\imath}\maxidx}
\right)
\left(  \frac{1}{\lambda_{\maxidx}}-\lambda_{\maxidx}\right)  \\
&=-\sum_{i}\frac{1}{\lambda_{i}}\hat{R}_{i\overline{\maxidx}i\overline{\maxidx}}
  +\lambda_{\maxidx}^{2}\sum_{i}\frac{1}{\lambda_{i}}\hat{R}_{i\overline{\maxidx}i\overline{\maxidx}}\\
&\qquad+\frac{1}{\lambda_{\maxidx}^{2}}\sum_{i}\lambda_{i}\hat{R}_{\overline{\imath}%
\maxidx\overline{\imath}\maxidx}-\sum_{i}\lambda_{i}\hat{R}_{\overline{\imath}\maxidx\overline{\imath}\maxidx}.
\end{align*}

If $\maxidx=1$ then the term
\[
\lambda_{1}^{2} \sum_{i}\frac{1}{\lambda_{i}}\hat{R}_{i\bar{1}i\bar{1}}
\]
dominates. Recall (here we use Corollary \ref{thetamaxprinciple}) 
\[
\det DT=\lambda_{1}\cdots\lambda_{n}
=\frac{e^{-2\theta}\rho}{\bar{\rho}}
\leq\Lambda_{0}
\]
so
\[
\Lambda_{0}\geq\lambda_{1}(\lambda_{n})^{\,n-1}
\]
and hence
\[
\frac{1}{\lambda_{n}}
\geq\left(  \frac{\lambda_{1}}{\Lambda_{0}}\right)^{\!\frac{1}{n-1}}.
\]

Thus
\[
\lambda_{1}^{2}\sum_{i}\frac{1}{\lambda_{i}}\hat{R}_{i\bar{1}i\bar{1}}
\geq
\lambda_{1}^{2}\left(  \frac{\lambda_{1}}{\Lambda_{0}}\right)^{\!\frac{1}{n-1}}
\hat{R}_{n\bar{1}n\bar{1}}
\geq
\kappa S_{\maxidx\maxidx}^{2+\frac{1}{n-1}}-C.
\]

Of the remaining terms,
\[
\frac{1}{\lambda_{\maxidx}^{2}}\sum_{i}\lambda_{i}\hat{R}_{\overline{\imath}\maxidx\overline{\imath}\maxidx}
\]
is positive, and the other two are bounded by $\sum_{i}S_{ii}.$

If $\maxidx=n,$ then
\[
\det DT=\lambda_{1}\cdots\lambda_{n}
=\frac{e^{-2\theta}\rho}{\bar{\rho}}
\geq\frac{1}{\Lambda_{0}}
\]
so
\[
\lambda_{1}\geq
\frac{1}{\Lambda_{0}(\lambda_{n})^{\,n-1}}
\]
and
\begin{align*}
\frac{1}{\lambda_{n}^{2}}\sum_{i}\lambda_{i}\hat{R}_{\bar{\imath}n\bar{\imath}n}
&\geq
\frac{1}{\lambda_{n}^{2}}
\frac{1}{\Lambda_{0}(\lambda_{n})^{\,n-1}}
\hat{R}_{\bar{1}n\bar{1}n}\\
&\geq
\kappa S_{\maxidx\maxidx}^{2+\frac{1}{n-1}}-C.
\end{align*}

\end{proof}

\begin{claim}\label{a67}
\begin{equation*}
\hat{R}(F_{i},F_{\maxidx},F_{i},KF_{\maxidx})
= -\frac{\lambda_{\maxidx}}{\lambda_{i}} \hat{R}_{i\overline{\maxidx} i\overline{\maxidx}}
  + \frac{\lambda_{i}}{\lambda_{\maxidx}} \hat{R}_{\overline{i} \maxidx \overline{i} \maxidx}.
\end{equation*}
\end{claim}
\begin{proof}
    
This can be verified by expanding 
\begin{eqnarray*}
&&\hat{R}\!\left(
\frac{1}{\sqrt{\lambda_{i}}}E_{i}+\sqrt{\lambda_{i}}E_{\bar{i}},
\frac{1}{\sqrt{\lambda_{\maxidx}}}E_{\maxidx}+\sqrt{\lambda_{\maxidx}}E_{\overline{\maxidx}},
\frac{1}{\sqrt{\lambda_{i}}}E_{i}+\sqrt{\lambda_{i}}E_{\bar{i}},
\frac{1}{\sqrt{\lambda_{\maxidx}}}E_{\maxidx}-\sqrt{\lambda_{\maxidx}}E_{\overline{\maxidx}}
\right)
\\
&=& \hat{R}\!\left(
\frac{1}{\sqrt{\lambda_{i}}}E_{i},\sqrt{\lambda_{\maxidx}}E_{\overline{\maxidx}},
\frac{1}{\sqrt{\lambda_{i}}}E_{i},-\sqrt{\lambda_{\maxidx}}E_{\overline{\maxidx}}
\right)
\\
&&+ \hat{R}\!\left(
\frac{1}{\sqrt{\lambda_{i}}}E_{i},\sqrt{\lambda_{\maxidx}}E_{\overline{\maxidx}},
\sqrt{\lambda_{i}}E_{\bar{i}},\frac{1}{\sqrt{\lambda_{\maxidx}}}E_{\maxidx}
\right)
\\
&&+ \hat{R}\!\left(
\sqrt{\lambda_{i}}E_{\bar{i}},\frac{1}{\sqrt{\lambda_{\maxidx}}}E_{\maxidx},
\frac{1}{\sqrt{\lambda_{i}}}E_{i},-\sqrt{\lambda_{\maxidx}}E_{\overline{\maxidx}}
\right)
\\
&&+ \hat{R}\!\left(
\sqrt{\lambda_{i}}E_{\bar{i}},\frac{1}{\sqrt{\lambda_{\maxidx}}}E_{\maxidx},
\sqrt{\lambda_{i}}E_{\bar{i}},\frac{1}{\sqrt{\lambda_{\maxidx}}}E_{\maxidx}
\right)
\\
&=& -\frac{\lambda_{\maxidx}}{\lambda_{i}} \hat{R}_{i\overline{\maxidx} i\overline{\maxidx}}
    + \hat{R}_{i\overline{\maxidx}\bar{i}\maxidx}
    - \hat{R}_{\bar{i}\maxidx i\overline{\maxidx}}
    + \frac{\lambda_{i}}{\lambda_{\maxidx}} \hat{R}_{\bar{i}\maxidx\bar{i}\maxidx}
\\
&=& -\frac{\lambda_{\maxidx}}{\lambda_{i}} \hat{R}_{i\overline{\maxidx} i\overline{\maxidx}}
    + \frac{\lambda_{i}}{\lambda_{\maxidx}} \hat{R}_{\bar{i}\maxidx\bar{i}\maxidx}.
\end{eqnarray*}

\end{proof}

\begin{claim} \label{annoying}%
\[
2\hat{S}\left(  \nabla_{F_{\maxidx}}^{\perp}n\left(  \nabla\psi\right)  ^{\perp}
,F_{\maxidx}\right)  \leq\left\Vert F_{\maxidx}\right\Vert _{S}\left\Vert \left(
\hat{\nabla}_{F_{\maxidx}}n\left(  \nabla\psi\right)  ^{\perp}\right)  ^{\perp}
\right\Vert _{\hat{S}}.
\]
\end{claim}

\begin{proof}
Unwinding
\[
-n\left(  \nabla\psi\right)  ^{\perp}
=\left(  n\left(  \nabla\psi\right)  \cdot KF_{i}\right)  g^{ij}KF_{j}
\]
and
\begin{align*}
&\hat{\nabla}_{F_{\maxidx}}\left[\left(  n\left(  \nabla\psi\right)  \cdot KF_{i}\right)
g^{ij}KF_{j}\right]\\
& =F_{\maxidx}\left(  \left(  n\left(  \nabla\psi\right)  \cdot KF_{i}\right)
g^{ij}\right)  KF_{j}\\
& \quad +\left(  n\left(  \nabla\psi\right)  \cdot KF_{i}\right)  g^{ij}K\hat
{\nabla}_{F_{\maxidx}}F_{j}.
\end{align*}

Projecting this to the normal direction, we note that $\hat{\nabla}_{F_{\maxidx}}F_{j}$
is already normal at the point, so $K$ projects that to the tangential direction,
and this term doesn't survive the projection. We are left with
\[
\hat{S}\left(  \nabla_{F_{\maxidx}}^{\perp}n\left(  \nabla\psi\right)  ^{\perp}
,F_{\maxidx}\right) =-
F_{\maxidx}\left(  \left(  n\left(  \nabla\psi\right)  \cdot KF_{i}\right)
g^{ij}\right)  \hat{S}\left(  KF_{j},F_{\maxidx}\right).
\]
We have taken diagonalizations so that only the $j=\maxidx$ term survives, so suffice to control
\[
F_{\maxidx}\left(  n\left(  \nabla\psi\right)  \cdot KF_{\maxidx}\right)  \hat{S}\left(
KF_{\maxidx},F_{\maxidx}\right).
\]

Next compute 
\begin{equation}
F_{\maxidx}\left(  n\left(  \nabla\psi\right)  \cdot KF_{\maxidx}\right)
=\hat{\nabla}_{F_{\maxidx}} n\left(  \nabla\psi\right)  \cdot K F_{\maxidx}
+ n\left(  \nabla\psi\right)\cdot K\hat{\nabla}_{F_{\maxidx}}F_{\maxidx}.
\label{thisone}
\end{equation}
To get control of
\[
\hat{\nabla}_{F_{\maxidx}}F_{\maxidx}
\]
observe that
\begin{equation}
\left(  \hat{\nabla}_{F_{\maxidx}}F_{\maxidx}\right)  ^{\perp}=A(F_{\maxidx},F_{\maxidx}) \label{seethe}
\end{equation}
is tensorial, so we may compute using
\begin{align*}
F_{\maxidx} &=\frac{1}{\sqrt{\lambda_{\maxidx}}}\widetilde{F}_{\maxidx}\text{ at }x,\\
\left(  \hat{\nabla}_{F_{\maxidx}}F_{\maxidx}\right)  ^{\perp}
&=A(F_{\maxidx},F_{\maxidx})=\frac{1}{\lambda_{\maxidx}}A(\widetilde{F}_{\maxidx},\widetilde{F}_{\maxidx})\\
&=-\frac{1}{\lambda_{\maxidx}}\left(  A(\widetilde{F}_{\maxidx},\widetilde{F}_{\maxidx})\cdot
K\widetilde{F}_{k}\right)  \widetilde{g}^{kk}K\widetilde{F}_{k}\\
&=- \frac{1}{\lambda_{\maxidx}}\left(  A(\widetilde{F}_{\maxidx},\widetilde{F}_{k})\cdot
K\widetilde{F}_{\maxidx}\right)  \widetilde{g}^{kk}K\widetilde{F}_{k}.
\end{align*}

Recall that
\[
\widetilde{F}_{i}=E_{i}+T_{i}^{\bar{s}}E_{\bar{s}}.
\]
Compute directly
\begin{align*}
A(\widetilde{F}_{\maxidx},\widetilde{F}_{k})
&=\hat{\nabla}_{\left(  E_{\maxidx}+T_{\maxidx}^{\bar{s}}E_{\bar{s}}\right)}
\left(  E_{k}+T_{k}^{\bar{s}}E_{\bar{s}}\right) \\
&=\Gamma_{\maxidx\,k}^{\alpha}E_{\alpha}
  +T_{\maxidx}^{\bar{s}}\Gamma_{\bar{s}k}^{\alpha}E_{\alpha}
  +T_{k\,\maxidx}^{\bar{s}}E_{\bar{s}}
  +T_{k}^{s}\Gamma_{\bar{s}\,\maxidx}^{\alpha}E_{\alpha}
  +T_{k}^{\bar{p}}(E_{\bar{p}}T_{\maxidx}^{\bar{s}})E_{\bar{s}}
  +T_{\maxidx}^{\bar{s}}T_{k}^{\bar{p}}\Gamma_{\bar{p}\bar{s}}^{\alpha}E_{\alpha}.
\end{align*}

After simplification,
\[
A(\widetilde{F}_{\maxidx},\widetilde{F}_{k})
=\Gamma_{\maxidx\,k}^{m}E_{m}+T_{k\,\maxidx}^{\bar{s}}E_{\bar{s}}.
\]

Thus
\begin{align}
A(\widetilde{F}_{\maxidx},\widetilde{F}_{k})\cdot K\widetilde{F}_{\maxidx}
&=\left(
\Gamma_{\maxidx\,k}^{m}E_{m}+T_{k\,\maxidx}^{\bar{s}}E_{\bar{s}}
\right)\cdot(E_{\maxidx}-\lambda_{\maxidx}E_{\overline{\maxidx}}) \notag\\
&=T_{k\,\maxidx}^{\maxidx}\left(-\frac12 c_{\overline{\maxidx}\,\maxidx}\right)
  -\lambda_{\maxidx}\Gamma_{\maxidx\,k}^{\maxidx}\left(-\frac12 c_{\overline{\maxidx}\,\maxidx}\right) \notag\\
&=-\frac12\left(T_{k\,\maxidx}^{\maxidx}
   -\lambda_{\maxidx}\Gamma_{\maxidx\,k}^{\maxidx}\right)
\label{cope}
\end{align}

Now differentiate
\[
\frac{S(\widetilde{F}_{\maxidx},\widetilde{F}_{\maxidx})}
{h(\widetilde{F}_{\maxidx},\widetilde{F}_{\maxidx})}
\]
in any $E_{k}$ direction, using 

\[
S(\widetilde{F}_{\maxidx},\widetilde{F}_{\maxidx})
=m_{\maxidx\maxidx}
+\bar{m}_{\bar{s}\bar{p}}T_{\maxidx}^{\bar{s}}T_{\maxidx}^{\bar{p}},
\qquad
h(\widetilde{F}_{\maxidx},\widetilde{F}_{\maxidx})
=-c_{\maxidx\bar{s}} T_{\maxidx}^{\bar{s}},
\]
(recalling we're using special coordinates type II (recall Corollary \ref{IIcoords})  and evaluating at the point) gives
\begin{align*}
E_{k}S(\widetilde{F}_{\maxidx},\widetilde{F}_{\maxidx})
 &=\bar{m}_{\overline{\maxidx}\,\overline{\maxidx},k}\lambda_{\maxidx}^{2}+2 \lambda_{\maxidx}T_{\maxidx\,k}^{\maxidx},\\
E_{k}h(\widetilde{F}_{\maxidx},\widetilde{F}_{\maxidx})
 &=-c_{\maxidx\overline{\maxidx}\,k}\lambda_{\maxidx}+T_{\maxidx\,k}^{\bar{\maxidx}}.
\end{align*}

Using that the maximum occurs at this point,
\[
h(\widetilde{F}_{\maxidx},\widetilde{F}_{\maxidx})
E_{k}S(\widetilde{F}_{\maxidx},\widetilde{F}_{\maxidx})
=
S(\widetilde{F}_{\maxidx},\widetilde{F}_{\maxidx})
E_{k}h(\widetilde{F}_{\maxidx},\widetilde{F}_{\maxidx})
\]

we obtain
\[
\left(  2\lambda_{\maxidx}-(1+\lambda^2_{\maxidx})\right)  T_{\maxidx\,k}^{\maxidx}
=
-(1+\lambda^2_{\maxidx})c_{\maxidx\overline{\maxidx}\,k}\lambda_{\maxidx}
-\bar{m}_{\overline{\maxidx}\,\overline{\maxidx},k}\lambda_{\maxidx}^{3},
\]

hence
\[
T_{\maxidx\,k}^{\maxidx}
=
\frac{-\lambda_{\maxidx}\left(
(1+\lambda^2_{\maxidx})c_{\maxidx\overline{\maxidx}\,k}-\bar{m}_{\overline{\maxidx}\,\overline{\maxidx},k}\lambda_{\maxidx}^{2}
\right)}
{-(1-\lambda_{\maxidx})^2}.
\]

From (\ref{cope}) and (\ref{seethe}) 
\begin{align*}
\left(  \hat{\nabla}_{F_{\maxidx}}F_{\maxidx}\right)  ^{\perp}
&=\frac{1}{\lambda_{\maxidx}}
  \left(  \frac{1}{2}(T_{k\,\maxidx}^{\maxidx}-\lambda_{\maxidx}\Gamma_{\maxidx\,k}^{\maxidx})\right)
  \frac{1}{\lambda_{k}}K\widetilde{F}_{k}\\
&=\left(
  \frac{1}{2}\frac{1}{\lambda_{\maxidx}}T_{k\,\maxidx}^{\maxidx}
  -\frac{1}{2}\Gamma_{\maxidx\,k}^{\maxidx}
 \right)
 \frac{(E_{k}-\lambda_{k}E_{\bar{k}})}{\lambda_{k}}\\
&=\left(
  \frac{1}{2}
   \frac{\lambda_{\maxidx}\left(
(1+\lambda^2_{\maxidx})c_{\maxidx\overline{\maxidx}\,k}-\bar{m}_{\overline{\maxidx}\,\overline{\maxidx},k}\lambda_{\maxidx}^{2}
\right)}
{(1-\lambda_{\maxidx})^2}
  -\frac{1}{2}\Gamma_{\maxidx\,k}^{\maxidx}
 \right)
 \frac{(E_{k}-\lambda_{k}E_{\bar{k}})}{\lambda_{k}}.
\end{align*}

Thus (we may assume that $\lambda_n<\frac{1}{2}$))
\[
\left\vert n\left(  \nabla\psi\right)  \cdot 
K\hat{\nabla}_{F_{\maxidx}}F_{\maxidx}\right\vert
\leq
\Vert nD\psi\Vert_{\hat{S}}
C\left(  1+ \frac{\lambda_1}{\lambda_n}\right)
\leq
C\left(  S_{\maxidx\maxidx}^{2}+1\right).
\]

Finally, recall (\ref{thisone})
\[
\hat{S}\left(
\left[  \hat{\nabla}_{F_{\maxidx}}
\left(  n(\nabla\psi)\right)^{\perp}\right]^{\perp},
F_{\maxidx}\right)
\leq
\Vert nD\psi\Vert_{\hat{S}}\Vert F_{\maxidx}\Vert_{\hat{S}}^{2}
+
C\left(  S_{\maxidx\maxidx}^{2}+1\right).
\]

\end{proof}

We conclude that there is a universal bound on $DT$ as shown below.

\begin{proposition}
Suppose that $(x,T(x))$ locally represents the generalized mean curvature flow. 
Assume that the flow remains in a compact set $Z$ that avoids the cut locus and on which the cross-curvature condition is positive. 
Then \[
\left\Vert DT\right\Vert =\sup_{m(V,V)=1}\left\Vert DT(V)\right\Vert _{\bar
{m}}\leq C(Z,\kappa,c,\rho,\bar{\rho},F_{0}).
\]
\end{proposition}

\begin{proof}
At the initial time there is an initial bound on $DT.$ \ Observe that by taking
normal coordinates for $m$ and diagonalizing 
\[
\widetilde{F}_{i}=E_{i}+\lambda_{i}E_{\bar{\imath}}%
\]
we get \
\[
\sup_{m(V,V)=1}\left\Vert DT(V)\right\Vert _{\bar{m}}^{2}\leq\lambda_{1}^{2}.
\]
Recall that
\[
\frac{1}{\lambda_{n}}\geq\left(  \frac{\lambda_{1}}{\Lambda_{0}}\right)
^{\frac{1}{n-1}}%
\]
so if $\lambda_{1}$ is large we will have
\[
S_{\maxidx\maxidx}=\max\left\{  \frac{1}{\lambda_{n}},\lambda_{1}\right\}  >C_{2}%
\]
big enough so that
\[
C_{1}S_{\maxidx\maxidx }^{2}-\kappa S_{\maxidx\maxidx}^{2+\frac{1}{n-1}}<0
\]
in which case, the ratio given by (\ref{def:ratio}) must be decreasing at the maximum
point by the previous Proposition. An upper bound on the ratio gives an upper bound on $\lambda_1$.\ 
\end{proof}

\section{Higher Regularity and Convergence } 

\subsection{Preliminary Estimates}

For the remainder, we assume the submanifold is flowing under a cross-curvature condition.  

Recalling \eqref{def:cexp}, we note that locally there exists a scalar function $u$ satisfying
\[
Du + D_x c(x,T(x)) = 0.
\]

\begin{claim}
Let $u$ be as in (\ref{def:cexp}). \ The GMCF is the projection of a
vertical flow described by
\[
u_{t}=-2\theta.
\]

\end{claim}

\begin{proof}
Since $\theta$ is geometric, we can choose any coordinates for $\bar{M}$ and
verify this at that point. \ We compute the vertical change of $T$ by
differentiating \
\begin{align*}
u_{i}+c_{i}(x,T(x))  & =0\\
u_{it}+c_{i\bar{s}}(x,T(x))T_{t}^{\bar{s}}  & =0.
\end{align*}
So at the origin, in special coordinates, we have
\[
T_{t}^{\bar{s}}=-2\theta_{s}.
\]
That is for the vertical flow
\[
\frac{d}{dt}(x,T(x,t))=(0,-2D\theta).
\]
Project this to the normal direction. \
\begin{align*}
(0,D\theta)^{\perp}  & =-\left(  (0,-2D\theta)\cdot K\tilde{F}_{i}\right)
\tilde{g}^{ij}K\tilde{F}_{j}\\
& =\theta_{i}\tilde{g}^{ij}(KF_{j})\\
& =K\left(  \theta_{i}\tilde{g}^{ij}F_{j}\right)  \\
& =K\nabla\theta\\
& =(\vec{H}-n\left(  \nabla\psi\right)  ^{\perp}).
\end{align*}

\end{proof}

Observe that differentiating, we get
\begin{equation}
u_{ij}+c_{ij}+c_{i\bar{s}}(x,T(x))T_{j}^{\bar{s}}=0 \label{c2c1}
\end{equation}
or
\[
\det(u_{ij}+c_{ij})=\det(-c_{i\bar{s}})\det T_{j}^{\bar{s}}.
\]
Recall
\begin{align*}
2\theta & =\ln\rho-\ln\bar{\rho}-\ln\det DT\\
& =\ln\rho-\ln\bar{\rho}-\ln\det(u_{ij}+c_{ij})+\ln\det(-c_{i\bar{s}})
\end{align*}
so plugging in, we get
\begin{align*}
u_{t}  & =-2\theta.\\
& =\ln\det(u_{ij}+c_{ij})-\ln\rho+\ln\bar{\rho}-\ln\det(-c_{i\bar{s}}).
\end{align*}

Thus $u$ locally solves a Monge-Amp\`{e}re type equation, and enjoys $C^{2}$
estimates (from the estimate on $DT$ and (\ref{c2c1})) in choices of coordinates. On a compact manifold, every domain is
interior, and by the Evans-Krylov theorem, we conclude that $u$ enjoys $C^{2,\alpha}$
estimates. Then, by Schauder theory, we obtain uniform bounds of
all orders.

At this stage, we have uniform $C^{1,\alpha}$ bounds for $T$ up to time $t_{0}$. 
Because the coefficients of the parabolic system are uniformly parabolic and H\"older continuous, Schauder theory applies and yields bounds on derivatives of all orders. In particular, the flow extends past any time $t_{0}$, and global $C^{k,\alpha}$ bounds hold for all $k$.

Once we have fixed a set of charts, interior estimates will apply along the flow. 
We have shown:
\begin{claim}
Suppose that $S$ is bounded above along the flow on $[0,t_0)$.  Then all higher-order derivatives exist and the flow extends past $t_0$.
\end{claim}

\subsection{Harnack Inequality and Exponential Convergence} 

Recalling claim \ref{thetalution}, we note that $\theta$ satisfies

$$\theta_t -\Delta_{-n\psi} \theta=0$$
where $\Delta_{-n\psi}$ is the weighted Laplacian.   Note that this flow tracks the change of $\theta$ as $x$ evolves under the normal flow. This is different from the vertical flow, which produces a similar weighted operator. Note that the vertical and normal evolution expressions for a quantity will always differ by a gradient term for that quantity.


Given the bounds on third derivatives of $T$, we obtain a lower bound on the weighted Ricci curvature associated with the operator $\Delta_{-n\psi}$. We are therefore exactly in the setting of \cite[Prop.~2.7]{AK}. The argument given there applies directly, and we conclude the same by defining

\[
f=\log\left(  \theta+c\right)
\]
for an appropriate $c$ and
\[
F=t\left(  \left\Vert \nabla_{g}f\right\Vert _{g}^{2}-\alpha\partial
_{t}f\right).
\]
So we get (for $\alpha=2)$
\[
LF+2\langle\nabla f,\nabla F\rangle\geq\frac{1}{t}\left(  C_{1}F^{2}%
-F-C_{2}t^{2}+C_{3}t\left\Vert \nabla f\right\Vert _{g}^{2}F\right)  .
\]
Applying this at the maximum for $F$ we get
\[
C_{1}F^{2}-F-C_{2}t^{2}+C_{3}t\left\Vert \nabla f\right\Vert _{g}^{2}F\leq0
\]
that is
\[
C_{1}F^{2}-F-C_{2}t^{2}\leq0
\]
or
\[
F(x,t)\leq\frac{1+\sqrt{1+4C_{1}C_{2}t^{2}}}{2C_{1}}\leq\tilde{C}_{1}%
+\tilde{C}_{2}t.
\]
That is, for positive $t$
\[
\left(  \left\vert \nabla f\right\vert ^{2}-2f_{t}\right)  \leq\frac{\tilde
{C}_{1}}{t}+\tilde{C}_{2}%
\]
so we have
\[
f_{t}\geq\frac{\left\vert \nabla f\right\vert ^{2}}{2}-\frac{1}{2}\left(
\frac{\tilde{C}_{1}}{t}+\tilde{C}_{2}\right).
\]
Now consider a path $\gamma:[0,1]\rightarrow L\times\lbrack t_{1},t_{2}]$ such
that
\begin{align*}
\gamma(0) &  =(y,t_{2})\\
\gamma(1) &  =\left(  x,t_{1}\right).
\end{align*}
Assume that the path projects to a geodesic in the metric on $L$ at $t_{1}$ with
constant speed. \ Then
\[
f(x,t_{1})-f\left(  y,t_{2}\right)  \leq\int_{0}^{1}\left\{  \left\vert \nabla
f(\gamma(s),s)\right\vert d_{(t_{1})}(x,y)+\left(  t_{1}-t_{2}\right)
f_{t}(\gamma(s),s)\right\}  ds.
\]
As $t_{1}<t_{2}$, we have
\begin{align*}
\left(  t_{1}-t_{2}\right)  f_{t}(\gamma(s),s) &  \leq\left(  t_{1}%
-t_{2}\right)  \left(  \left\vert \nabla f\right\vert ^{2}-\left(
\frac{\tilde{C}_{1}}{t}+\tilde{C}_{2}\right)  \right)  \\
&  =-\left(  t_{2}-t_{1}\right)  \left(  \left\vert \nabla f\right\vert
^{2}-\left(  \frac{\tilde{C}_{1}}{t}+\tilde{C}_{2}\right)  \right)  \\
&  =\left(  t_{2}-t_{1}\right)  \left(  -\left\vert \nabla f\right\vert
^{2}+\left(  \frac{\tilde{C}_{1}}{t}+\tilde{C}_{2}\right)  \right)
\end{align*}
and 
\[
f(x,t_{1})-f\left(  y,t_{2}\right)  \leq\int_{0}^{1}\left\{  \left\vert \nabla
f(\gamma(s),s)\right\vert d_{(t_{1})}(x,y)+\left(  t_{2}-t_{1}\right)  \left(
-\left\vert \nabla f\right\vert ^{2}+\left(  \frac{\tilde{C}_{1}}{t(s)}%
+\tilde{C}_{2}\right)  \right)  \right\}  ds.
\]
Using
\[
\left\vert \nabla f(\gamma(s),s)\right\vert d_{(t_{1})}(x,y)\leq\left\vert
\nabla f(\gamma(s),s)\right\vert ^{2}\left(  t_{2}-t_{1}\right)
+\frac{d_{(t_{1})}^{2}(x,y)}{4\left(  t_{2}-t_{1}\right)  }%
\]
 we have%
\begin{align*}
f(x,t_{1})-f\left(  y,t_{2}\right)   &  \leq\int_{0}^{1}\left\{
\frac{d_{(t_{1})}^{2}(x,y)}{4\left(  t_{2}-t_{1}\right)  }+\left\vert \nabla
f(\gamma(s),s)\right\vert ^{2}\left(  t_{2}-t_{1}\right)  -\left\vert \nabla
f\right\vert ^{2}\left(  t_{2}-t_{1}\right)  +\left(  t_{2}-t_{1}\right)
\left(  \frac{\tilde{C}_{1}}{t(s)}+\tilde{C}_{2}\right)  \right\}  ds\\
&  \leq\int_{0}^{1}\left\{  \frac{d_{(t_{1})}^{2}(x,y)}{4\left(  t_{2}%
-t_{1}\right)  }+\left(  t_{2}-t_{1}\right)  \left(  \frac{\tilde{C}_{1}%
}{t(s)}+\tilde{C}_{2}\right)  \right\}  ds\\
&  \leq\frac{d_{(t_{1})}^{2}(x,y)}{4\left(  t_{2}-t_{1}\right)  }+\tilde
{C}_{2}\left(  t_{2}-t_{1}\right)  -\int_{t_{2}}^{t_{1}}\frac{\tilde{C}_{1}%
}{t}dt\\
&  =\frac{d_{(t_{1})}^{2}(x,y)}{4\left(  t_{2}-t_{1}\right)  }+\tilde{C}%
_{2}\left(  t_{2}-t_{1}\right)  +\tilde{C}_{1}\log\left(  \frac{t_{2}}{t_{1}%
}\right)
\end{align*}
noting that
\begin{align*}
\frac{dt}{ds} &  =-\left(  t_{2}-t_{1}\right)  ds\\
t(0) &  =t_{2}\\
t(1) &  =t_{1}.
\end{align*}
This gives us that
\[
\log\left(  \theta+c\right)  (x,t_{1})\leq\log\left(  \theta+c\right)
(y,t_{2})+\frac{d_{(t_{1})}^{2}(x,y)}{4\left(  t_{2}-t_{1}\right)  }+\tilde
{C}_{2}\left(  t_{2}-t_{1}\right)  +\tilde{C}_{1}\log\left(  \frac{t_{2}%
}{t_{1}}\right)
\]
which we can exponentiate to get
\[
\left(  \theta+c\right)  (x,t_{1})\leq\left(  \theta+c\right)  (y,t_{2}%
)\exp\left(  \frac{d_{(t_{1})}^{2}(x,y)}{4\left(  t_{2}-t_{1}\right)  }%
+\tilde{C}_{2}\left(  t_{2}-t_{1}\right)  \right)  \left(  \frac{t_{2}}{t_{1}%
}\right)  ^{\tilde{C}_{1}}.
\]
If we assume that $\theta>0$ we can let $c\rightarrow0$%
\[
\theta(x,t_{1})\leq\theta(y,t_{2})\exp\left(  \frac{\text{diam}_{g_{t_{1}}%
}^{2}\text{(}M)}{4\left(  t_{2}-t_{1}\right)  }+\tilde{C}_{2}\left(
t_{2}-t_{1}\right)  \right)  \left(  \frac{t_{2}}{t_{1}}\right)  ^{\tilde
{C}_{1}}.
\]
Thus for $t_{1}\geq1$ and $t_{2}=t_{1}+1/2$ we have%
\[
\theta(x,t_{1})\leq C_{3}\theta(y,t_{2})
\]
for all $x,y.$ \ 

Now we may apply the standard convergence argument - see for example \cite[Section 7.1]{KSW}.
\ \ It follows that the oscillation of $\theta$ decreases exponential to a
constant. \ The convergence of the flow is in all orders, so we conclude that the limit as
$t\rightarrow\infty$ must be a manifold satisfying
\[
D\theta\equiv0
\]
that is
\[
\vec{H}-n\left(  \nabla\psi\right)  ^{\perp}=0.
\]

\bibliographystyle{amsalpha}
\bibliography{ot}

@misc{behrndt2008,
      title={Lagrangian mean curvature flow in almost Kaehler-Einstein manifolds}, 
      author={Tapio Behrndt},
      year={2008},
      eprint={0812.4256},
      archivePrefix={arXiv},
      primaryClass={math.DG},
      url={https://arxiv.org/abs/0812.4256}, 
}

@incollection {SmoczykSurvey,
    AUTHOR = {Smoczyk, Knut},
     TITLE = {Mean curvature flow in higher codimension: introduction and
              survey},
 BOOKTITLE = {Global differential geometry},
    SERIES = {Springer Proc. Math.},
    VOLUME = {17},
     PAGES = {231--274},
 PUBLISHER = {Springer, Heidelberg},
      YEAR = {2012},
      ISBN = {978-3-642-22841-4; 978-3-642-22842-1},
   MRCLASS = {53C44},
  MRNUMBER = {3289845},
MRREVIEWER = {Karsten\ T.\ Gimre},
       DOI = {10.1007/978-3-642-22842-1\_9},
       URL = {https://doi.org/10.1007/978-3-642-22842-1_9},
}

@misc{S96,
      title={A canonical way to deform a Lagrangian submanifold}, 
      author={Knut Smoczyk},
      year={1996},
      eprint={dg-ga/9605005},
      archivePrefix={arXiv},
      primaryClass={dg-ga},
      url={https://arxiv.org/abs/dg-ga/9605005}, 
}

@incollection {Bryant,
    AUTHOR = {Bryant, Robert L.},
     TITLE = {Minimal {L}agrangian submanifolds of {K}\"ahler-{E}instein
              manifolds},
 BOOKTITLE = {Differential geometry and differential equations ({S}hanghai,
              1985)},
    SERIES = {Lecture Notes in Math.},
    VOLUME = {1255},
     PAGES = {1--12},
 PUBLISHER = {Springer, Berlin},
      YEAR = {1987},
      ISBN = {3-540-17849-X},
   MRCLASS = {53C42 (53C55 58E15)},
  MRNUMBER = {895393},
MRREVIEWER = {Wolfgang\ Ziller},
       DOI = {10.1007/BFb0077676},
       URL = {https://doi.org/10.1007/BFb0077676},
}

@article {Smo99,
    AUTHOR = {Smoczyk, Knut},
     TITLE = {Harnack inequality for the {L}agrangian mean curvature flow},
   JOURNAL = {Calc. Var. Partial Differential Equations},
  FJOURNAL = {Calculus of Variations and Partial Differential Equations},
    VOLUME = {8},
      YEAR = {1999},
    NUMBER = {3},
     PAGES = {247--258},
      ISSN = {0944-2669,1432-0835},
   MRCLASS = {53C44 (53D12 58E15)},
  MRNUMBER = {1688545},
       DOI = {10.1007/s005260050125},
       URL = {https://doi.org/10.1007/s005260050125},
}

@article {Gold,
    AUTHOR = {Goldstein, Edward},
     TITLE = {Calibrated fibrations},
   JOURNAL = {Comm. Anal. Geom.},
  FJOURNAL = {Communications in Analysis and Geometry},
    VOLUME = {10},
      YEAR = {2002},
    NUMBER = {1},
     PAGES = {127--150},
      ISSN = {1019-8385,1944-9992},
   MRCLASS = {53C38 (32Q25)},
  MRNUMBER = {1894143},
MRREVIEWER = {Peng\ Lu},
       DOI = {10.4310/CAG.2002.v10.n1.a6},
       URL = {https://doi.org/10.4310/CAG.2002.v10.n1.a6},
}

@article{W2025,
  title={Minimal Lagrangian submanifolds of weighted Kim--McCann metrics},
  author={Warren, Micah W},
  journal={Canadian Mathematical Bulletin},
  pages={1--16},
  year={2025},
  publisher={Canadian Mathematical Society}
}

@article {BLMR,
    AUTHOR = {Brendle, Simon and L\'eger, Flavien and McCann, Robert J. and
              Rankin, Cale},
     TITLE = {A geometric approach to apriori estimates for optimal
              transport maps},
   JOURNAL = {J. Reine Angew. Math.},
  FJOURNAL = {Journal f\"ur die Reine und Angewandte Mathematik. [Crelle's
              Journal]},
    VOLUME = {817},
      YEAR = {2024},
     PAGES = {251--266},
      ISSN = {0075-4102,1435-5345},
   MRCLASS = {53C40 (35Q49 49Q22 53C50)},
  MRNUMBER = {4834886},
       DOI = {10.1515/crelle-2024-0071},
       URL = {https://doi.org/10.1515/crelle-2024-0071},
}

@article {LV,
    AUTHOR = {L\'eger, Flavien and Vialard, Fran\c cois-Xavier},
     TITLE = {A geometric {L}aplace method},
   JOURNAL = {Pure Appl. Anal.},
  FJOURNAL = {Pure and Applied Analysis},
    VOLUME = {5},
      YEAR = {2023},
    NUMBER = {4},
     PAGES = {1041--1080},
      ISSN = {2578-5885,2578-5893},
   MRCLASS = {53C15 (49Q22 53B12)},
  MRNUMBER = {4680531},
       DOI = {10.2140/paa.2023.5.1041},
       URL = {https://doi.org/10.2140/paa.2023.5.1041},
}

@article {CSS,
    AUTHOR = {Chursin, Mykhaylo and Sch\"afer, Lars and Smoczyk, Knut},
     TITLE = {Mean curvature flow of space-like {L}agrangian submanifolds in
              almost para-{K}\"ahler manifolds},
   JOURNAL = {Calc. Var. Partial Differential Equations},
  FJOURNAL = {Calculus of Variations and Partial Differential Equations},
    VOLUME = {41},
      YEAR = {2011},
    NUMBER = {1-2},
     PAGES = {111--125},
      ISSN = {0944-2669,1432-0835},
   MRCLASS = {53C44 (53C15 53D12)},
  MRNUMBER = {2782799},
MRREVIEWER = {Paolo\ Piccinni},
       DOI = {10.1007/s00526-010-0355-x},
       URL = {https://doi.org/10.1007/s00526-010-0355-x},
}

@inproceedings{dazord1981geometrie,
  title={Sur la g{\'e}om{\'e}trie des sous-fibr{\'e}s et des feuilletages lagrangiens},
  author={Dazord, Pierre},
  booktitle={Annales scientifiques de l'{\'E}cole Normale Sup{\'e}rieure},
  volume={14},
  number={4},
  pages={465--480},
  year={1981}
}

@article{joyce2001lectures,
  title={Lectures on Calabi-Yau and special Lagrangian geometry},
  author={Joyce, Dominic},
  journal={arXiv preprint math/0108088},
  year={2001}
}

@article{bryant2000second,
  title={Second order families of special Lagrangian 3-folds},
  author={Bryant, Robert L},
  journal={arXiv preprint math.DG/0007128},
  year={2000}
}

@article {Para,
    AUTHOR = {Cruceanu, V. and Fortuny, P. and Gadea, P. M.},
     TITLE = {A survey on paracomplex geometry},
   JOURNAL = {Rocky Mountain J. Math.},
  FJOURNAL = {The Rocky Mountain Journal of Mathematics},
    VOLUME = {26},
      YEAR = {1996},
    NUMBER = {1},
     PAGES = {83--115},
      ISSN = {0035-7596,1945-3795},
   MRCLASS = {53C56 (53C15)},
  MRNUMBER = {1386154},
MRREVIEWER = {Sorin\ Dragomir},
       DOI = {10.1216/rmjm/1181072105},
       URL = {https://doi.org/10.1216/rmjm/1181072105},
}

@article{Wood,
  title={Singularities of Lagrangian Mean Curvature Flow},
  author={Albert Wood},
  journal={ Doctoral thesis, UCL (University College London)},
  year={2020}
}

@article {SWAsianJMath,
    AUTHOR = {Smoczyk, Knut and Wang, Mu-Tao},
     TITLE = {Generalized {L}agrangian mean curvature flows in symplectic
              manifolds},
   JOURNAL = {Asian J. Math.},
  FJOURNAL = {Asian Journal of Mathematics},
    VOLUME = {15},
      YEAR = {2011},
    NUMBER = {1},
     PAGES = {129--140},
      ISSN = {1093-6106,1945-0036},
   MRCLASS = {53C44 (53D12 53D35)},
  MRNUMBER = {2786468},
MRREVIEWER = {John\ Urbas},
       DOI = {10.4310/AJM.2011.v15.n1.a7},
       URL = {https://doi.org/10.4310/AJM.2011.v15.n1.a7},
}

@incollection {B,
    AUTHOR = {Behrndt, Tapio},
     TITLE = {Generalized {L}agrangian mean curvature flow in {K}\"ahler
              manifolds that are almost {E}instein},
 BOOKTITLE = {Complex and differential geometry},
    SERIES = {Springer Proc. Math.},
    VOLUME = {8},
     PAGES = {65--79},
 PUBLISHER = {Springer, Heidelberg},
      YEAR = {2011},
      ISBN = {978-3-642-20299-5},
   MRCLASS = {53C44},
  MRNUMBER = {2964468},
MRREVIEWER = {Yuan\ Yuan},
       DOI = {10.1007/978-3-642-20300-8\_3},
       URL = {https://doi.org/10.1007/978-3-642-20300-8_3},
}

@article{paletal,
  title={Wasserstein mirror gradient flow as the limit of the Sinkhorn algorithm},
  author={Deb, Nabarun and Kim, Young-Heon and Pal, Soumik and Schiebinger, Geoffrey},
  journal={arXiv preprint arXiv:2307.16421},
  year={2023}
}

@article{berman2020sinkhorn,
  title={The Sinkhorn algorithm, parabolic optimal transport and geometric Monge--Amp{\`e}re equations},
  author={Berman, Robert J},
  journal={Numerische Mathematik},
  volume={145},
  number={4},
  pages={771--836},
  year={2020},
  publisher={Springer}
}

@article{joyce2015conjectures,
  title={Conjectures on Bridgeland stability for Fukaya categories of Calabi--Yau manifolds, special Lagrangians, and Lagrangian mean curvature flow},
  author={Joyce, Dominic},
  journal={EMS Surveys in Mathematical Sciences},
  volume={2},
  number={1},
  pages={1--62},
  year={2015}
}

@article{thomas2001special,
  title={Special Lagrangians, stable bundles and mean curvature flow},
  author={Thomas, Richard P and Yau, S-T},
  journal={COMMUN ANAL GEOM 10 ( 5 ) 1075 - 1113},
  year={2002}
}

@article {LS,
    AUTHOR = {Li, Guanghan and Salavessa, Isabel M. C.},
     TITLE = {Mean curvature flow of spacelike graphs},
   JOURNAL = {Math. Z.},
  FJOURNAL = {Mathematische Zeitschrift},
    VOLUME = {269},
      YEAR = {2011},
    NUMBER = {3-4},
     PAGES = {697--719},
      ISSN = {0025-5874,1432-1823},
   MRCLASS = {53C44 (53C21)},
  MRNUMBER = {2860260},
MRREVIEWER = {Andreas\ Savas-Halilaj},
       DOI = {10.1007/s00209-010-0768-4},
       URL = {https://doi.org/10.1007/s00209-010-0768-4},
}

@article {MTW,
    AUTHOR = {Ma, Xi-Nan and Trudinger, Neil S. and Wang, Xu-Jia},
     TITLE = {Regularity of potential functions of the optimal
              transportation problem},
   JOURNAL = {Arch. Ration. Mech. Anal.},
  FJOURNAL = {Archive for Rational Mechanics and Analysis},
    VOLUME = {177},
      YEAR = {2005},
    NUMBER = {2},
     PAGES = {151--183},
      ISSN = {0003-9527,1432-0673},
   MRCLASS = {35J60 (35B65 35J20 49Q20)},
  MRNUMBER = {2188047},
MRREVIEWER = {Luisa\ Moschini},
       DOI = {10.1007/s00205-005-0362-9},
       URL = {https://doi.org/10.1007/s00205-005-0362-9},
}

@article {KSW,
    AUTHOR = {Kim, Young-Heon and Streets, Jeffrey and Warren, Micah},
     TITLE = {Parabolic optimal transport equations on manifolds},
   JOURNAL = {Int. Math. Res. Not. IMRN},
  FJOURNAL = {International Mathematics Research Notices. IMRN},
      YEAR = {2012},
    NUMBER = {19},
     PAGES = {4325--4350},
      ISSN = {1073-7928,1687-0247},
   MRCLASS = {58J35 (35K55 35R01 53C20)},
  MRNUMBER = {2981711},
MRREVIEWER = {Francesco\ Rossi},
       DOI = {10.1093/imrn/rnr188},
       URL = {https://doi.org/10.1093/imrn/rnr188},
}

@book {Joyce,
    AUTHOR = {Joyce, Dominic D.},
     TITLE = {Riemannian holonomy groups and calibrated geometry},
    SERIES = {Oxford Graduate Texts in Mathematics},
    VOLUME = {12},
 PUBLISHER = {Oxford University Press, Oxford},
      YEAR = {2007},
     PAGES = {x+303},
      ISBN = {978-0-19-921559-1},
   MRCLASS = {53C38 (53C29)},
  MRNUMBER = {2292510},
MRREVIEWER = {Spiro\ Karigiannis},
}

@article {AK,
    AUTHOR = {Abedin, Farhan and Kitagawa, Jun},
     TITLE = {Exponential convergence of parabolic optimal transport on
              bounded domains},
   JOURNAL = {Anal. PDE},
  FJOURNAL = {Analysis \& PDE},
    VOLUME = {13},
      YEAR = {2020},
    NUMBER = {7},
     PAGES = {2183--2204},
      ISSN = {2157-5045,1948-206X},
   MRCLASS = {49Q22 (35K96 58J35)},
  MRNUMBER = {4175823},
MRREVIEWER = {Heming\ Jiao},
       DOI = {10.2140/apde.2020.13.2183},
       URL = {https://doi.org/10.2140/apde.2020.13.2183},
}

@article {KM,
    AUTHOR = {Kim, Young-Heon and McCann, Robert J.},
     TITLE = {Continuity, curvature, and the general covariance of optimal
              transportation},
   JOURNAL = {J. Eur. Math. Soc. (JEMS)},
  FJOURNAL = {Journal of the European Mathematical Society (JEMS)},
    VOLUME = {12},
      YEAR = {2010},
    NUMBER = {4},
     PAGES = {1009--1040},
      ISSN = {1435-9855,1435-9863},
   MRCLASS = {49Q20 (58E25)},
  MRNUMBER = {2654086},
MRREVIEWER = {Lorenzo\ Brasco},
       DOI = {10.4171/JEMS/221},
       URL = {https://doi.org/10.4171/JEMS/221},
}

@article {KMW,
    AUTHOR = {Kim, Young-Heon and McCann, Robert J. and Warren, Micah},
     TITLE = {Pseudo-{R}iemannian geometry calibrates optimal
              transportation},
   JOURNAL = {Math. Res. Lett.},
  FJOURNAL = {Mathematical Research Letters},
    VOLUME = {17},
      YEAR = {2010},
    NUMBER = {6},
     PAGES = {1183--1197},
      ISSN = {1073-2780},
   MRCLASS = {49Q20 (49Q05 53C38)},
  MRNUMBER = {2729641},
MRREVIEWER = {Luigi\ De Pascale},
       DOI = {10.4310/MRL.2010.v17.n6.a16},
       URL = {https://doi.org/10.4310/MRL.2010.v17.n6.a16},
}

@article {MM,
    AUTHOR = {Mantegazza, Carlo and Martinazzi, Luca},
     TITLE = {A note on quasilinear parabolic equations on manifolds},
   JOURNAL = {Ann. Sc. Norm. Super. Pisa Cl. Sci. (5)},
  FJOURNAL = {Annali della Scuola Normale Superiore di Pisa. Classe di
              Scienze. Serie V},
    VOLUME = {11},
      YEAR = {2012},
    NUMBER = {4},
     PAGES = {857--874},
      ISSN = {0391-173X,2036-2145},
   MRCLASS = {58J35 (35K59 35R01)},
  MRNUMBER = {3060703},
MRREVIEWER = {Boubaker-Khaled\ Sadallah},
}

@book {MMMT,
    AUTHOR = {Mitrea, Dorina and Mitrea, Irina and Mitrea, Marius and
              Taylor, Michael},
     TITLE = {The {H}odge-{L}aplacian},
    SERIES = {De Gruyter Studies in Mathematics},
    VOLUME = {64},
      NOTE = {Boundary value problems on Riemannian manifolds},
 PUBLISHER = {De Gruyter, Berlin},
      YEAR = {2016},
     PAGES = {ix+516},
      ISBN = {978-3-11-048266-9; 978-3-11-048438-0; 978-3-11-048339-0},
   MRCLASS = {58J32 (35R01 58Jxx)},
  MRNUMBER = {3586566},
MRREVIEWER = {Thomas\ Schick},
       DOI = {10.1515/9783110484380},
       URL = {https://doi.org/10.1515/9783110484380},
}

@book {CC,
	TITLE = {Fully nonlinear elliptic equations},
	AUTHOR = {Caffarelli, Luis A and Cabr{\'e}, Xavier},
	VOLUME = {43},
	YEAR = {1995},
	PUBLISHER = {American Mathematical Soc.},
	EPRINT = {hep-ph/9609357}
}

@book {Lee,
    AUTHOR = {Lee, John M.},
     TITLE = {Introduction to smooth manifolds},
    SERIES = {Graduate Texts in Mathematics},
    VOLUME = {218},
   EDITION = {Second},
 PUBLISHER = {Springer, New York},
      YEAR = {2013},
     PAGES = {xvi+708},
      ISBN = {978-1-4419-9981-8},
   MRCLASS = {58-01 (53-01 57-01)},
  MRNUMBER = {2954043},
}

\end{document}